\documentclass[journal,twoside,web]{ieeecolor}
\usepackage{generic}
\usepackage{cite}
\usepackage{amsmath,amssymb,amsfonts}
\usepackage{algorithmic}
\usepackage{graphicx}
\usepackage{algorithm,algorithmic}
\usepackage{hyperref}
\hypersetup{hidelinks=true}
\usepackage{textcomp}
\usepackage{tikz}
\usetikzlibrary{arrows.meta,positioning}
\usepackage{etoolbox}
\hypersetup{hypertexnames=false}

\newtheorem{theorem}{Theorem}
\newtheorem{lemma}{Lemma}
\newtheorem{corollary}{Corollary}
\newtheorem{proposition}{Proposition}

\newtheorem{assumption}{Assumption}
\newtheorem{remark}{Remark}
\def\BibTeX{{\rm B\kern-.05em{\sc i\kern-.025em b}\kern-.08em
    T\kern-.1667em\lower.7ex\hbox{E}\kern-.125emX}}
\begin{document}
\title{A Backstepping Framework for Unconstrained Accelerated Optimization Algorithms}
\author{Song Chen, Jiaxu Liu, and Chao Xu, \IEEEmembership{Senior Member, IEEE}
\thanks{This work was supported by  the Key Research and Development Project of China National Tobacco Corporation  under Grant 110202402018 and the National Natural Science Foundation of China (No. 62573382, No. 12171431, No. 62373323, No. 72342025). \emph{(Corresponding author: Chao Xu, Jiaxu Liu.)}}
\thanks{Song Chen is with the Department of Mathematics, National University of Singapore, Singapore, Singapore (e-mail: song.chen@nus.edu.sg). }
\thanks{Jiaxu Liu is with the School of Mathematical Sciences, Zhejiang University, Hangzhou, Zhejiang 310027, China. (e-mail: jiaxuliu@zju.edu.cn)}
\thanks{Chao Xu is with the State Key Laboratory of Industrial Control Technology, the Institute of Cyber-Systems and Control, College of Control Science and Engineering, Zhejiang University, Hangzhou, Zhejiang 310027, China, and also with Huzhou Institute of Zhejiang University, Huzhou, Zhejiang 313000, China.  (e-mail: cxu@zju.edu.cn).}}

\maketitle

\begin{abstract}
This paper introduces a control-theoretic perspective on unconstrained optimization algorithms using the backstepping methods. We model the optimization process as an augmented strict-feedback system given by 
$\dot{x}_1 = x_2$, $\dot{x}_2 = u$, and $\dot{z} = q(x_1,z)$, 
with a regulated output $y = \nabla f(x_1)$. This formulation recasts the development of unconstrained optimization algorithms as a feedback control problem, where the goal is to design the input $u$ to ensure $y(t) \to 0$. By employing backstepping, we recursively synthesize the actual feedback law $u$ after initially selecting a virtual control for $x_1$. For convex objective functions, we develop a general synthesis framework for augmented strict-feedback systems and specialize it to the standard strict-feedback case. This unified framework successfully recovers the constant-parameter Nesterov flow and the proportional-integral-derivative (PID) accelerated optimizer as direct corollaries. We further establish that, given a fixed virtual control, the universal second-step law is inverse optimal with respect to an induced outer-tracking problem. This reveals that the optimality of the control law is conditionally dependent on the target manifold prescribed by the virtual control, rather than holding globally across all possible backstepping designs. Finally, we formulate a formal optimal-backstepping theorem that elevates this optimality principle to the virtual-control stage by solving a reduced Hamilton--Jacobi--Bellman problem. These contributions collectively yield a robust and general backstepping-driven paradigm for the analysis and design of continuous-time unconstrained optimization algorithms.
\end{abstract}




\section{Introduction}
\label{sec:introduction}

\IEEEPARstart{C}{ontrol} concepts increasingly shape the design of modern learning systems, not only their analysis. Across contemporary artificial intelligence, control-inspired mechanisms have become architectural and algorithmic primitives: PID feedback ideas have motivated real-time semantic-segmentation architectures such as PIDNet \cite{xu2023pidnet}, explicit conditional-control modules have been introduced into large text-to-image diffusion models through ControlNet \cite{zhang2023adding}, state-space modeling principles rooted in systems and signal processing now underpin scalable sequence architectures such as S4 \cite{gu2022efficiently} and Mamba \cite{gu2024mamba}, and PID-type feedback has been used to accelerate deep-learning optimization \cite{chen2024accelerated}. These examples span perception, generation, sequence modeling, and training dynamics, yet they convey a common message: control can lead to better learning mechanisms, not merely better post hoc explanations. Taken together, they point to a broader and promising research direction—control for better learning—in which feedback, memory, damping, and state-space structure are treated as design primitives for improving modern learning systems.

Within this broad agenda, the present paper focuses on control and optimization. Iterative optimization schemes can often be interpreted as feedback interconnections between an algorithmic update rule and an oracle that provides gradient, stochastic-gradient, or proximal information. From a discrete-time perspective, many optimization algorithms can be represented as feedback interconnections between an algorithmic update rule and an oracle that provides gradient, stochastic-gradient, or proximal information. To make this concrete, consider the standard gradient descent algorithm, $x_{k+1} = x_k - \eta \nabla f(x_k)$. From a control-theoretic perspective, this iterative process is a simple closed-loop system: the optimization variable $x_k$ acts as the system's ``state," the linear update mechanism $x_{k+1} = x_k + \eta u_k$ serves as the dynamic ``plant" (an integrator), and the gradient computation $u_k = -\nabla f(x_k)$ acts as a nonlinear feedback controller that observes the current state and applies a corrective force and this viewpoint has been considered in \cite{hauswirth2024optimization}. Furthermore, classical momentum and accelerated methods \cite{polyak1964some,nesterov269method} have therefore been studied using integral-quadratic-constraint (IQC) analyses and algorithm-design tools \cite{lessard2016analysis,fazlyab2018analysis}, dissipativity-based analyses \cite{hu2017dissipativity,lessard2022analysis}, and Lyapunov arguments for accelerated schemes \cite{wilson2021lyapunov,sanz2021connections}. In addition, proportional-integral-derivative (PID)-type control ideas have been introduced into stochastic optimization, showing that PID feedback can be used to accelerate training dynamics in deep neural networks \cite{wang2020pid}. These developments suggest that control theory provides not only a language for describing iterative optimization algorithms, but also a systematic framework for analyzing convergence rates, robustness margins, and algorithmic parameter choices.

The connection between optimization and control is even more explicit in continuous time, where algorithms are viewed as discretizations or sampled implementations of underlying dynamical systems. Accelerated and momentum methods \cite{polyak1964some,nesterov269method} admit ordinary differential equation (ODE) models \cite{su2016differential}, variational interpretations \cite{wibisono2016variational}, continuous-time Lyapunov explanations of acceleration \cite{wilson2021lyapunov,luo2022differential}, and advanced stability characterizations such as fixed-time stable gradient flows \cite{garg2020fixed}. More refined dynamical descriptions further reveal the distinct mechanisms behind Nesterov acceleration and heavy-ball dynamics, for example through high-resolution differential equations \cite{shi2022understanding} and dynamical-systems interpretations of Nesterov acceleration \cite{muehlebach2019dynamical}. Beyond purely dissipative flows, Hamiltonian and symplectic perspectives emphasize the roles of energy, damping, and structure-preserving discretization in optimization dynamics \cite{maddison2018hamiltonian,muehlebach2021optimization}. Moreover, control-theoretic synthesis has been actively used to explicitly design continuous-time optimization dynamics. For unconstrained and hybrid dynamics, synthesis tools have generated continuous frameworks \cite{kolarijani2019continuous} and output-feedback-based continuous-time PID optimization algorithms \cite{chen2024accelerated,liu2024output,liu2023proposal,liu2026distributed}. Furthermore, constrained optimization problems have been actively reformulated as dynamic regulation tasks using dynamic Karush-Kuhn-Tucker (KKT) controllers \cite{jokic2009constrained}, control barrier functions \cite{allibhoy2023control}, feedback linearization techniques \cite{zhang2025constrained}, and explicit feedback control of Lagrange multipliers \cite{cerone2025new}, alongside formal analyses of the regularity properties of these optimization-based controllers \cite{mestres2025regularity}. Overall, the discrete-time viewpoint highlights feedback interconnection, robustness certification, and controller-like tuning of iterative schemes, whereas the continuous-time viewpoint clarifies the energy, damping, and trajectory structures that underlie acceleration optimization dynamics.

However, a fundamental limitation of the existing literature lies in its predominantly retrospective approach: one first adopts a predefined optimization algorithm or continuous-time model, and then develops a Lyapunov certificate or feedback interpretation to characterize the algorithm’s dynamic behavior. This observation motivates the central question of this paper: \textbf{\textit{Can the synthesis of unconstrained optimization algorithms be formulated not as a retrospective analysis of a pre-existing method, but as a proactive nonlinear control design problem?}}

To answer this question, we develop a constructive backstepping framework for continuous-time optimization dynamics. The motivation for adopting backstepping is twofold. First, many momentum-based \cite{su2016differential},\cite{shi2022understanding} and accelerated optimization models \cite{chen2024accelerated},\cite{attouch2018fast} naturally have a cascaded structure, in which the position variable is driven by a momentum-like state and the momentum dynamics are shaped by a gradient-dependent input. This structure is closely aligned with the strict-feedback form underlying backstepping \cite{krstic1995nonlinear}. Second, backstepping is intrinsically constructive: it recursively builds a stabilizing feedback law together with an associated control Lyapunov function, rather than analyzing a prescribed algorithm after the fact \cite{vaidyanathan2020backstepping}. This constructive versatility has been successfully demonstrated across a wide variety of advanced control paradigms, including output-feedback and event-triggered designs \cite{wang2021event}, making it particularly suitable for proactive algorithm synthesis. To be specific, in this paper, we cast the optimization process as the regulation of the second-order plant
\begin{equation}
    \dot{x}_1 = x_2, \qquad \dot{x}_2 = u,
\end{equation}
possibly augmented with additional dynamic extensions. For unconstrained convex optimization, the condition $\nabla f(x_1)=0$ characterizes optimality under standard assumptions. We therefore take $\nabla f(x_1)$ as the regulated output and design the dynamics so that
$$
    \nabla f(x_1(t)) \to 0 .
$$
Under this formulation, an optimization algorithm is not assumed in advance. It is generated through the backstepping procedure: one first designs a virtual control that defines a desired descent manifold for the position dynamics, and then synthesizes the actual control input that drives the momentum state toward this manifold.

Furthermore, shifting from retrospective analysis to constructive synthesis naturally invites a deeper investigation into algorithm optimality. Classical nonlinear inverse-optimal control demonstrates that stabilizing feedback laws constructed via backstepping are inherently inverse optimal with respect to some meaningful cost functionals \cite{freeman1996inverse}, a property that extends even to modern prescribed-time and stochastic stabilization settings \cite{li2021stochastic}. However, translating this property to backstepping-generated optimization algorithms raises a subtle but critical issue: \textbf{\textit{if the synthesized control law is inverse optimal, what exactly is it optimizing---the overall algorithmic trajectory, or merely the tracking of a specific intermediate path?}}

Our analysis shows that the latter is the correct interpretation. For a fixed virtual control $\alpha$, the actual control $u^\star$ is inverse optimal only for the induced outer-tracking problem associated with the manifold selected by $\alpha$. Hence the inverse optimality of the universal second-step law is conditional rather than global: it is a statement about the optimal tracking of a preassigned target manifold, not an optimality statement over all admissible backstepping designs. Genuine algorithm-level optimality must therefore be lifted to the virtual-control stage.

To the best of our knowledge, the explicit recasting of continuous-time unconstrained optimization algorithm synthesis as a backstepping design problem for a strict-feedback system---together with the separation between outer inverse-optimal tracking and inner virtual-control optimality---has not been formulated in the existing optimization and nonlinear-control literature cited above. This is the central viewpoint advanced in the present paper.

The main contributions are as follows.

\begin{enumerate}
\item We establish an explicit bridge between unconstrained optimization algorithms, strict-feedback systems, and recursive backstepping design. Under this bridge, the task of constructing an optimization algorithm is reformulated as the task of designing a stabilizing feedback law for regulating the output $\nabla f(x_1)$ of a second-order plant together with its augmented strict-feedback extensions.
\item We derive a general backstepping synthesis for augmented strict-feedback systems and a corresponding strict-feedback specialization. This framework recovers a constant-parameter Nesterov flow \cite{nesterov269method},\cite{su2016differential} and the recent proportional-integral-derivative accelerated optimizer (PIDAO) \cite{chen2024accelerated} as direct theorem-level instances.
\item We prove that, for a fixed virtual control, the universal second-step law is inverse optimal only for an induced outer-tracking problem. This provides a precise theoretical statement of why the optimality of $u^\star$ should be interpreted as ``given-manifold optimal tracking'' rather than as a globally optimal backstepping design.
\item We formulate a genuine optimal-backstepping principle by imposing optimality at the virtual-control stage through a reduced Hamilton--Jacobi--Bellman problem. This yields a two-layer interpretation: the virtual control is optimal for a prescribed reduced problem, while the actual control is inverse optimal for the induced outer-tracking problem.
\end{enumerate}

The remainder of the paper is organized as follows. Section~\ref{sec:problem-formulation} formulates the problem. Section~\ref{sec:general-backstepping} develops the general backstepping framework for gradient regulation. Furthermore, Section~\ref{sec: direct theorem applications} presents direct theorem-level realizations corresponding to the constant-parameter Nesterov flow and the PIDAO flow. Section~\ref{sec:inverse-optimality} establishes the fixed-virtual-control inverse-optimality result and the optimal-backstepping principle. Section~\ref{sec:examples} presents representative examples generated by the two-layer optimality framework. Section~\ref{sec:conclusion} concludes.

\section{Problem Formulation}\label{sec:problem-formulation}

Consider the unconstrained optimization problem
\[
\min_{x\in\mathbb{R}^n} f(x),
\]
where $f:\mathbb{R}^n\to\mathbb{R}$ is the objective function. To solve this problem efficiently, one often employs accelerated first-order methods, such as the heavy-ball method \cite{polyak1964some} and Nesterov's accelerated method \cite{nesterov269method}; see also the continuous-time interpretation in \cite{su2016differential}. In continuous time, such accelerated algorithms can be viewed as closed-loop realizations of the second-order plant
\begin{equation}
\begin{aligned}
\dot x_1 &= x_2, \\
\dot x_2 &= u,
\end{aligned}
\label{eq:strict-feedback-plant}
\end{equation}
where $x_1, x_2 \in \mathbb{R}^n$. For example, the feedback law $u = -a x_2 - \nabla f(x_1)$ with $a>0$ produces the heavy-ball dynamics \cite{polyak1964some}. More generally, continuous-time optimization dynamics can be interpreted as feedback laws for \eqref{eq:strict-feedback-plant}, so that the design of an optimization algorithm becomes a controller-design problem: construct $u$ such that the output $y := \nabla f(x_1)$ converges to zero. This viewpoint is also consistent with recent control-inspired optimization constructions such as PIDAO \cite{chen2024accelerated}.

Motivated by this perspective, we first study the more general augmented strict-feedback system
\[
\dot x_1 = x_2, \qquad \dot x_2 = u, \qquad \dot z = q(x_1,z),
\]
where $z\in\mathbb{R}^m$ is an auxiliary state and $q:\mathbb{R}^n\times\mathbb{R}^m\to\mathbb{R}^m$ is assumed to be locally Lipschitz. The auxiliary state is introduced to capture additional controller or algorithmic dynamics beyond the basic second-order plant. For example, one may introduce a filtered-gradient state $z=\xi$ satisfying $\dot \xi = -\beta \xi + \nabla f(x_1)$ with $\beta>0$, so that $\xi$ acts as a filtered gradient state. When the auxiliary dynamics are absent, this system reduces to the standard second-order plant \eqref{eq:strict-feedback-plant}. Let
\[
\begin{aligned}
f^\star &:= \inf_{x\in\mathbb{R}^n} f(x), \\
\phi(x_1) &:= f(x_1) - f^\star.
\end{aligned}
\]
The control objective is to design $u$ such that
\begin{equation}
\nabla f(x_1(t)) \to 0 \qquad \text{as } t\to\infty,
\label{eq:gradient-regulation-goal}
\end{equation}
without requiring convergence of $x_1(t)$ to any prescribed minimizer. {\color{black}This formulation is natural for convex objectives that are not necessarily strongly convex, where stationarity is the fundamental requirement and the minimizer need not be unique.}

To achieve \eqref{eq:gradient-regulation-goal}, this paper adopts a backstepping design for the augmented strict-feedback system \cite{krstic1995nonlinear}. The construction proceeds in two steps.

\begin{enumerate}
\item \textit{Virtual-control design.} The pair $(x_1,z)$ is first viewed as a reduced subsystem with a virtual input $x_2$. One then selects a virtual control
\[
x_2 = \alpha(x_1,z),
\]
so that the reduced dynamics
\[
\dot x_1 = \alpha(x_1,z), \qquad \dot z = q(x_1,z)
\]
render the output $y=\nabla f(x_1)$ dissipative or contractive in a suitable Lyapunov/storage sense. The purpose of this step is to determine the target manifold and the first-step energy structure that encode the desired optimization behavior.

\item \textit{Actual-control design.} Since $x_2$ is a state rather than a directly assignable input, the actual control $u$ must be designed so that $x_2$ tracks the prescribed virtual control $\alpha(x_1,z)$. To this end, one introduces the backstepping error
\[
e := x_2-\alpha(x_1,z),
\]
and then constructs $u$ to stabilize the $e$-dynamics while preserving the first-step dissipation established above. The purpose of this step is to force the full closed-loop system toward the target manifold $e=0$ and thereby recover the desired asymptotic regulation $\nabla f(x_1(t))\to 0$.
\end{enumerate}

To ease the forthcoming analysis, we collect the standing assumptions and auxiliary lemmas used throughout the subsequent development.



\begin{assumption}\label{asm:convex-smooth-objective}
The objective function $f : \mathbb{R}^n \to \mathbb{R}$ is of class $C^2$, convex, and $L$-smooth, i.e.,
\[
\|\nabla f(x) - \nabla f(y)\| \le L\|x-y\|, \qquad \forall x,y\in\mathbb{R}^n,
\]
and is bounded from below: $f^\star > -\infty$.

Under Assumption~\ref{asm:convex-smooth-objective}, convexity and $L$-smoothness imply
\begin{equation}
0 \preceq \nabla^2 f(x_1) \preceq L I.
\label{eq:hessian-smoothness-bound}
\end{equation}
\end{assumption}

\begin{lemma}\cite{nesterov2013introductory}\label{lem:gradient-suboptimality}
Under Assumption~\ref{asm:convex-smooth-objective}, the bound
\[
\frac{1}{2L}\|\nabla f(x_1)\|^2 \le f(x_1) - f^\star = \phi(x_1)
\]
holds for all $x_1\in\mathbb{R}^n$.
\end{lemma}

\begin{lemma}[Barbalat's Lemma] \cite{khalil2002nonlinear}\label{lem:barbalat}
Let $h:[0,\infty)\to\mathbb{R}$ be uniformly continuous and satisfy $\int_0^\infty |h(t)|\,dt < \infty$. Then $h(t)\to 0$ as $t\to\infty$.
\end{lemma}

\begin{lemma}\label{lem:gradient-vanishing}
Suppose that along a trajectory of any system satisfying $\dot x_1 = x_2$, the following conditions hold:

\begin{enumerate}
\item $\phi(x_1(t))$ is bounded on $[0,\infty)$;
\item $x_2(t)$ is bounded on $[0,\infty)$;
\item $\int_0^\infty \|\nabla f(x_1(t))\|^2\,dt < \infty$.
\end{enumerate}

Then $\nabla f(x_1(t)) \to 0$ as $t\to\infty$.
\end{lemma}

\begin{proof}
Let $h(t) := \frac{1}{2}\|\nabla f(x_1(t))\|^2$. By Lemma~\ref{lem:gradient-suboptimality}, boundedness of $\phi(x_1(t))$ implies boundedness of $\nabla f(x_1(t))$. Moreover, we have
\[
\dot h
=
\nabla f(x_1(t))^\top \frac{d}{dt}\nabla f(x_1(t))
=
\nabla f(x_1(t))^\top \nabla^2 f(x_1(t))x_2.
\]
Using \eqref{eq:hessian-smoothness-bound}, we can derive
\[
\begin{aligned}
|\dot h|
&\le \|\nabla f(x_1(t))\|\,\|\nabla^2 f(x_1(t))\|\,\|x_2\| \\
&\le L\|\nabla f(x_1(t))\|\,\|x_2\|.
\end{aligned}
\]
Hence, $\dot h$ is bounded, and $h$ is uniformly continuous. Since
\[
\int_0^\infty h(t)\,dt
=
\frac{1}{2}\int_0^\infty \|\nabla f(x_1(t))\|^2\,dt
<
\infty,
\]
Lemma~\ref{lem:barbalat} yields $h(t)\to 0$, i.e., $\nabla f(x_1(t)) \to 0$.
\end{proof}

\section{A General Backstepping Framework for Gradient Regulation}\label{sec:general-backstepping}

We now develop a recursive construction in the spirit of classical backstepping for strict-feedback nonlinear systems \cite{krstic1995nonlinear}, but with the gradient output $\nabla f(x_1)$ rather than the plant state as the regulated quantity.

\begin{theorem}[General backstepping synthesis]\label{thm:general-backstepping}
Consider the augmented system
\begin{equation}
\begin{aligned}
\dot x_1 &= x_2, \\
\dot x_2 &= u, \\
\dot z   &= q(x_1,z),
\end{aligned}
\label{eq:augmented-strict-feedback}
\end{equation}
where $z\in\mathbb{R}^m$ is an auxiliary state and $q:\mathbb{R}^n\times\mathbb{R}^m\to\mathbb{R}^m$ is locally Lipschitz. The case with no auxiliary dynamics is included by taking $m=0$ and suppressing the $z$-dependence. Let a virtual control $\alpha(x_1,z)$ be selected and define the backstepping error
$e := x_2 - \alpha(x_1,z)$.
Assume that the following conditions hold:

\begin{enumerate}
\item There exists a continuously differentiable function $V_1(x_1,z)$ and constants $c_\phi>0$, $c_z>0$, $\sigma>0$, and $\lambda>0$ such that
\begin{equation}
V_1(x_1,z) \ge c_\phi \phi(x_1) + c_z\|z\|^2
\label{eq:storage-lower-bound}
\end{equation}
for all $(x_1,z)$, and such that along
\[
\dot x_1 = \alpha(x_1,z) + e, \qquad \dot z = q(x_1,z)
\]
one has
\begin{equation}
\dot V_1 = -\sigma\|\nabla f(x_1)\|^2 + \lambda \nabla f(x_1)^\top e.
\label{eq:first-step-dissipation}
\end{equation}

\item If $(\nabla f(x_1),z)$ is bounded, then so is $\alpha(x_1,z)$.
\end{enumerate}

We choose the actual control input as
\begin{equation}
u = \dot \alpha - c e - \lambda \nabla f(x_1), \qquad c>0,
\label{eq:backstepping-feedback-law}
\end{equation}
where $\dot\alpha$ denotes the total derivative of $\alpha(x_1,z)$ along \eqref{eq:augmented-strict-feedback}. Then the closed-loop system satisfies
\begin{equation}
\dot V = -\sigma\|\nabla f(x_1)\|^2 - c\|e\|^2, \qquad V := V_1 + \frac{1}{2}\|e\|^2,
\label{eq:closed-loop-dissipation}
\end{equation}
and consequently
\begin{equation}
\nabla f(x_1(t)) \to 0 \qquad \text{as } t\to\infty.
\label{eq:gradient-convergence}
\end{equation}
\end{theorem}

\begin{proof}
Since $e = x_2 - \alpha(x_1,z)$, one has $\dot e = u - \dot\alpha$, and under the feedback law \eqref{eq:backstepping-feedback-law},
\[
\dot e = -ce - \lambda \nabla f(x_1).
\]
Therefore, we can obtain
\[
\begin{aligned}
\dot V &= \dot V_1 + e^\top \dot e \\
&= \bigl(-\sigma\|\nabla f(x_1)\|^2 + \lambda \nabla f(x_1)^\top e\bigr) + e^\top(-ce-\lambda \nabla f(x_1)) \\
&= -\sigma\|\nabla f(x_1)\|^2 - c\|e\|^2,
\end{aligned}
\]
which proves \eqref{eq:closed-loop-dissipation}. Hence, $V(t)$ is nonincreasing and bounded from below. By \eqref{eq:storage-lower-bound}, we can derive
\[
c_\phi\phi(x_1(t)) + c_z\|z(t)\|^2 + \frac{1}{2}\|e(t)\|^2 \le V(t) \le V(0),
\]
which implies boundedness of $\phi(x_1(t))$, $z(t)$, and $e(t)$. By Lemma~\ref{lem:gradient-suboptimality}, $\nabla f(x_1(t))$ is bounded. By assumption, boundedness of $(\nabla f(x_1(t)),z(t))$ implies boundedness of $\alpha(x_1,z)$, hence $x_2(t)=e(t)+\alpha(x_1(t),z(t))$ is bounded. Moreover, from \eqref{eq:closed-loop-dissipation}, we have
\[
\int_0^\infty \|\nabla f(x_1(t))\|^2\,dt \le \frac{V(0)-\inf V}{\sigma} < \infty.
\]
Applying Lemma~\ref{lem:gradient-vanishing} to the projected $(x_1,x_2)$-trajectory yields \eqref{eq:gradient-convergence}.
\end{proof}

\begin{corollary}[Strict-feedback specialization of Theorem~\ref{thm:general-backstepping}]\label{cor:strict-feedback-specialization}
Consider the standard strict-feedback system
\begin{equation}
\dot x_1 = x_2, \qquad \dot x_2 = u.
\label{eq:standard-strict-feedback}
\end{equation}
Let a virtual control $\alpha(x_1)$ be selected and define $e := x_2 - \alpha(x_1)$. Assume that there exist a continuously differentiable function $V_1(x_1)$ and constants $c_\phi>0$, $\sigma>0$, and $\lambda>0$ such that
\[
V_1(x_1) \ge c_\phi\phi(x_1)
\]
for all $x_1$, and such that along
\[
\dot x_1 = \alpha(x_1) + e
\]
one has
\begin{equation}
\dot V_1 = -\sigma\|\nabla f(x_1)\|^2 + \lambda \nabla f(x_1)^\top e.
\label{eq:strict-first-step}
\end{equation}
Assume in addition that boundedness of $\nabla f(x_1)$ implies boundedness of $\alpha(x_1)$. If the actual control is chosen as
\[
u = \dot\alpha - c e - \lambda \nabla f(x_1), \qquad c>0,
\]
where $\dot\alpha = \frac{\partial \alpha}{\partial x_1}x_2$, then the closed-loop system satisfies
\[
\dot V = -\sigma\|\nabla f(x_1)\|^2 - c\|e\|^2, \qquad V := V_1 + \frac{1}{2}\|e\|^2,
\]
and consequently $\nabla f(x_1(t)) \to 0$ as $t\to\infty$.
\end{corollary}

\begin{proof}
This is the special case of Theorem~\ref{thm:general-backstepping} obtained by suppressing the auxiliary state $z$.
\end{proof}


\begin{remark}
Theorem~\ref{thm:general-backstepping} and Corollary~\ref{cor:strict-feedback-specialization} provide a structural framework rather than a unique algorithm. Their role is to separate the design into a first-step choice of the virtual control $\alpha$ and a second-step synthesis of the actual input $u$.

\begin{enumerate}
\item \textit{Choosing the virtual control.} The virtual control determines the target manifold
\[
\mathcal M_\alpha := \{(x_1,z,x_2): x_2=\alpha(x_1,z)\}.
\]
Its purpose is to shape the reduced dynamics
\[
\dot x_1=\alpha(x_1,z), \qquad \dot z=q(x_1,z),
\]
so that the first-step storage function $V_1$ satisfies the dissipation relation \eqref{eq:first-step-dissipation}. Different choices of $\alpha$ lead to different optimization flows. Typical examples include
\begin{itemize}
\item a gradient-type virtual control, e.g.,
\[
\alpha(x_1)=-k\nabla f(x_1);
\]
\item a normalized-gradient virtual control, e.g.,
\[
\alpha(x_1)=-k\frac{\nabla f(x_1)}{\varepsilon+\|\nabla f(x_1)\|}, \qquad \varepsilon>0;
\]
\item an adaptive or preconditioned gradient law, e.g.,
\[
\alpha(x_1,z)=-D(z)\nabla f(x_1),
\]
where $D(z)$ is a state-dependent positive-definite scaling matrix generated by auxiliary dynamics, in the spirit of adaptive-gradient methods.
\end{itemize}

\item \textit{Designing the actual control.} Once $\alpha$ has been fixed, the actual control $u$ no longer decides the target manifold; instead, it is chosen to drive the full system toward $\mathcal M_\alpha$. This is achieved through the backstepping error
\[
e=x_2-\alpha(x_1,z),
\]
which measures the deviation from the prescribed virtual behavior. The second-step feedback law \eqref{eq:backstepping-feedback-law} is constructed precisely so that $e$ is stabilized and the full state approaches the manifold $e=0$, while preserving the first-step dissipation encoded by $V_1$.
\end{enumerate}
\label{rem: target manifold}
\end{remark}

\section{Direct Theorem-Level Applications via Gradient-Type Virtual Controls}\label{sec: direct theorem applications}

Theorem~\ref{thm:general-backstepping} and Corollary~\ref{cor:strict-feedback-specialization} already encode a family of optimization dynamics once the virtual control has been specified. In particular, by choosing gradient-type virtual controls, one can directly obtain theorem-level realizations of both the constant-parameter Nesterov flow and the PIDAO flow. The first arises from the standard strict-feedback case covered by Corollary~\ref{cor:strict-feedback-specialization}, whereas the second arises from the augmented setting covered by Theorem~\ref{thm:general-backstepping}.

\begin{theorem}[Constant-parameter Nesterov flow]\label{thm:nesterov-flow}
Let Assumption~\ref{asm:convex-smooth-objective} hold. Consider the strict-feedback plant \eqref{eq:standard-strict-feedback}. Select the virtual control $\alpha(x_1) := -k_1 \nabla f(x_1)$ with $k_1>0$, and define the backstepping error
\begin{equation}
e := x_2 - \alpha(x_1) = x_2 + k_1 \nabla f(x_1).
\label{eq:nesterov-error}
\end{equation}
Choose the first-step Lyapunov function $V_1(x_1) := \phi(x_1)$. Then, along the first-step dynamics
\[
\dot x_1 = \alpha(x_1) + e,
\]
one has
\begin{equation}
\dot V_1 = -k_1\|\nabla f(x_1)\|^2 + \nabla f(x_1)^\top e.
\label{eq:nesterov-first-step}
\end{equation}
Hence the hypotheses of Corollary~\ref{cor:strict-feedback-specialization} are satisfied with $c_\phi=1$, $\sigma=k_1$, and $\lambda=1$. If the actual control is selected according to the backstepping law
\begin{equation}
u = \dot\alpha - k_2 e - \nabla f(x_1), \qquad k_2>0,
\label{eq:nesterov-feedback-law}
\end{equation}
then the explicit feedback law becomes
\begin{equation}
u = -k_2\bigl(x_2 + k_1\nabla f(x_1)\bigr) - \nabla f(x_1) - k_1\nabla^2f(x_1)x_2,
\label{eq:nesterov-explicit-law}
\end{equation}
The resulting second-order dynamics are
\begin{equation}
\ddot X + k_2\dot X + (1+k_1k_2)\nabla f(X) + k_1\nabla^2f(X)\dot X = 0,
\label{eq:nesterov-second-order}
\end{equation}
and the corresponding closed-loop trajectory satisfies
\begin{equation}
\nabla f(x_1(t)) \to 0 \qquad \text{as } t\to\infty.
\label{eq:nesterov-gradient-convergence}
\end{equation}
\end{theorem}

\begin{proof}
Since $x_2 = e-k_1\nabla f(x_1)$, one obtains
\[
\dot V_1 = \dot\phi = \nabla f(x_1)^\top x_2 = -k_1\|\nabla f(x_1)\|^2 + \nabla f(x_1)^\top e,
\]
which proves \eqref{eq:nesterov-first-step}. Thus the hypotheses of Corollary~\ref{cor:strict-feedback-specialization} are satisfied. Moreover,
\begin{equation}
\dot\alpha = -k_1\nabla^2f(x_1)x_2.
\label{eq:nesterov-alpha-derivative}
\end{equation}
Substituting \eqref{eq:nesterov-error} and \eqref{eq:nesterov-alpha-derivative} into \eqref{eq:nesterov-feedback-law} yields \eqref{eq:nesterov-explicit-law}, and therefore \eqref{eq:nesterov-second-order}. The asymptotic property \eqref{eq:nesterov-gradient-convergence} follows immediately from Corollary~\ref{cor:strict-feedback-specialization}.
\end{proof}

\medskip
Recently, Chen \textit{et al.} \cite{chen2024accelerated} introduced the PIDAO from a control perspective. The next theorem shows that its continuous-time dynamics arise directly from the augmented backstepping framework.

\begin{theorem}[PIDAO as a direct application of Theorem~\ref{thm:general-backstepping}]\label{thm:pidao-flow}
Let Assumption~\ref{asm:convex-smooth-objective} hold. Consider the augmented strict-feedback system
\[
\begin{aligned}
\dot x_1 &= x_2, \\
\dot x_2 &= u, \\
\dot \xi &= \nabla f(x_1).
\end{aligned}
\]
Choose the virtual control and the backstepping error as
\begin{equation}
\begin{aligned}
\alpha(x_1,\xi) &:= -k_d \nabla f(x_1) - \frac{k_i}{a}\xi, \\
e &:= x_2 - \alpha = x_2 + k_d \nabla f(x_1) + \frac{k_i}{a}\xi,
\end{aligned}
\label{eq:pidao-virtual-error}
\end{equation}
where $a>0$, $k_d>0$, and $k_i>0$. Define
\begin{equation}
\rho := k_p - a k_d - \frac{k_i}{a},
\label{eq:pidao-rho-balance}
\end{equation}
and assume that $\rho > 0$. Let the first-step Lyapunov function be chosen as
\begin{equation}
V_1(x_1,\xi) := \rho\left(\phi(x_1) + \frac{k_i}{2a}\|\xi\|^2\right).
\label{eq:pidao-storage}
\end{equation}
Then, along the first-step dynamics
\begin{equation}
\dot x_1 = \alpha(x_1,\xi) + e, \qquad \dot \xi = \nabla f(x_1),
\label{eq:pidao-first-step-system}
\end{equation}
one has
\begin{equation}
\dot V_1 = -\rho k_d\|\nabla f(x_1)\|^2 + \rho \nabla f(x_1)^\top e.
\label{eq:pidao-first-step}
\end{equation}
Hence the hypotheses of Theorem~\ref{thm:general-backstepping} are satisfied with $z=\xi$, $q(x_1,\xi)=\nabla f(x_1)$, $c_\phi=\rho$, $c_z=\rho k_i/(2a)$, $\sigma=\rho k_d$, and $\lambda=\rho$; moreover, boundedness of $(\nabla f(x_1),\xi)$ implies boundedness of $\alpha(x_1,\xi)$. If the actual control is selected according to \eqref{eq:backstepping-feedback-law} with $c=a$ and $\lambda=\rho$, then
\begin{equation}
u = \dot\alpha - a e - \rho \nabla f(x_1),
\label{eq:pidao-feedback-law}
\end{equation}
which yields the explicit feedback law
\begin{equation}
u = -a x_2 - k_p\nabla f(x_1) - k_i\xi - k_d\nabla^2 f(x_1)x_2,
\label{eq:pidao-explicit-law}
\end{equation}
The corresponding second-order dynamics are
\begin{equation}
\ddot X + a\dot X + k_p\nabla f(X) + k_i\int_0^t \nabla f(X(s))\,ds + k_d\frac{d}{dt}\nabla f(X) = 0.
\label{eq:pidao-second-order}
\end{equation}
Consequently,
\begin{equation}
\nabla f(x_1(t)) \to 0 \qquad \text{as } t\to\infty.
\label{eq:pidao-gradient-convergence}
\end{equation}
\end{theorem}

\begin{proof}
Differentiating \eqref{eq:pidao-storage} along \eqref{eq:pidao-first-step-system} gives
\[
\dot V_1
=
\rho \nabla f(x_1)^\top(\alpha+e) + \rho\frac{k_i}{a}\xi^\top \nabla f(x_1).
\]
Substituting $\alpha=-k_d \nabla f(x_1)-\frac{k_i}{a}\xi$ from \eqref{eq:pidao-virtual-error} yields
\[
\begin{aligned}
\dot V_1
&=
\rho \nabla f(x_1)^\top
\left(-k_d \nabla f(x_1)-\frac{k_i}{a}\xi+e\right) \\
&\quad + \rho\frac{k_i}{a}\xi^\top \nabla f(x_1) \\
&=
-\rho k_d\|\nabla f(x_1)\|^2 + \rho \nabla f(x_1)^\top e,
\end{aligned}
\]
which proves \eqref{eq:pidao-first-step}. Since
\begin{equation}
\dot\alpha = -k_d\nabla^2 f(x_1)x_2 - \frac{k_i}{a}\nabla f(x_1),
\label{eq:pidao-alpha-derivative}
\end{equation}
substituting \eqref{eq:pidao-virtual-error}, \eqref{eq:pidao-rho-balance}, and \eqref{eq:pidao-alpha-derivative} into \eqref{eq:pidao-feedback-law} yields
\[
\begin{aligned}
u
&=
-k_d\nabla^2 f(x_1)x_2 - \frac{k_i}{a}\nabla f(x_1) \\
&\quad - a\left(x_2 + k_d \nabla f(x_1) + \frac{k_i}{a}\xi\right)
 - \rho \nabla f(x_1) \\
&=
-a x_2 - \left(a k_d + \rho + \frac{k_i}{a}\right)\nabla f(x_1) \\
&\quad - k_i\xi - k_d\nabla^2 f(x_1)x_2 \\
&=
-a x_2 - k_p\nabla f(x_1) - k_i\xi \\
&\quad - k_d\nabla^2 f(x_1)x_2,
\end{aligned}
\]
which is \eqref{eq:pidao-explicit-law}. Equation~\eqref{eq:pidao-second-order} is the corresponding second-order form. The conclusion \eqref{eq:pidao-gradient-convergence} then follows directly from Theorem~\ref{thm:general-backstepping}.
\end{proof}

\section{Fixed-Virtual-Control Inverse Optimality and an Optimal Backstepping Principle}\label{sec:inverse-optimality}

The next results separate two layers of optimality. The first layer is an inverse-optimal interpretation akin to classical nonlinear inverse-optimal control \cite{freeman1996inverse}, \cite{krstic1998inverse}, but specialized here to the backstepping architecture induced by a fixed virtual control. The second layer lifts the optimality requirement to the virtual-control stage itself.

Following Remark~\ref{rem: target manifold}, once the virtual control $\alpha$ is fixed, the target manifold
\[
\mathcal M_\alpha := \{(x_1,x_2,z): x_2=\alpha(x_1,z)\}
\]
is fixed as well. The role of the actual input $u$ is then clear: it must drive the full system toward $\mathcal M_\alpha$, equivalently, it must regulate the outer error $e:=x_2-\alpha(x_1,z)$ while preserving the first-step dissipation generated by $(\alpha,V_1)$. At first sight, this role is purely stabilizing or tracking in nature. The point of the next result is that the universal second-step law does more than stabilize this outer error: once $\alpha$ has selected the target manifold, the same law also solves an induced optimal control problem for the outer dynamics.

Specifically, fix an admissible virtual control $\alpha(x_1,z)$ and a first-step storage function $V_1(x_1,z)$ satisfying the hypotheses of Theorem~\ref{thm:general-backstepping}. For an arbitrary admissible control input $u$, the transformed dynamics in the coordinates $(x_1,z,e)$ are
\[
\dot x_1 = \alpha(x_1,z) + e, \qquad \dot z = q(x_1,z), \qquad \dot e = u - \dot\alpha.
\]
Define
\[
\begin{aligned}
V &:= V_1 + \frac{1}{2}\|e\|^2, \\
\ell_\alpha(x_1,z,e,u)
&:= \sigma\|\nabla f(x_1)\|^2 + \frac{c}{2}\|e\|^2 \\
&\quad + \frac{1}{2c}\|u-\dot\alpha+\lambda \nabla f(x_1)\|^2.
\end{aligned}
\]
Along every admissible input $u$, the square-completion identity
\begin{equation}
\dot V + \ell_\alpha
=
\frac{1}{2c}\|u-\dot\alpha+\lambda \nabla f(x_1) + ce\|^2 \ge 0.
\label{eq:outer-square-completion}
\end{equation}
holds.

\begin{proposition}[Fixed-virtual-control inverse optimality]\label{prop:fixed-virtual-optimality}
With $\alpha$, $V_1$, $V$, and $\ell_\alpha$ defined above, define for each $T>0$ the Bolza functional
\[
J_T^\alpha(u)
:=
V(x_1(T),z(T),e(T))
+
\int_0^T \ell_\alpha(x_1,z,e,u)\,dt.
\]
Then the control
\[
u^\star := \dot\alpha - ce - \lambda \nabla f(x_1)
\]
is, for every $T>0$, the unique minimizer of the optimal control problem
\[
\min_u J_T^\alpha(u),
\]
where the minimization is over all admissible controls on $[0,T]$.
\end{proposition}

\begin{proof}
For an arbitrary admissible input $u$, one has
\[
\dot V = \dot V_1 + e^\top(u-\dot\alpha)
= -\sigma\|\nabla f(x_1)\|^2 + e^\top(u-\dot\alpha+\lambda \nabla f(x_1)),
\]
where the last identity follows from \eqref{eq:first-step-dissipation}. The functional $J_T^\alpha$ is the natural Bolza functional associated with the outer dynamics, because integrating the square-completion identity \eqref{eq:outer-square-completion} over $[0,T]$ yields
\begin{equation}
\begin{aligned}
J_T^\alpha(u)
&= V(x_1(T),z(T),e(T))  + \int_0^T \ell_\alpha(x_1,z,e,u)\,dt \\
&\ge V(x_1(0),z(0),e(0)).
\end{aligned}
\label{eq:outer-bolza-bound}
\end{equation}
Moreover, equality in \eqref{eq:outer-bolza-bound} holds if and only if the square term in \eqref{eq:outer-square-completion} vanishes identically, namely,
\begin{equation}
u = \dot\alpha - ce - \lambda \nabla f(x_1).
\label{eq:outer-optimal-control}
\end{equation}
Hence $u^\star$ attains the minimum value of $J_T^\alpha$. Since $\ell_\alpha$ is strictly convex in $u$, the minimizer is unique.
\end{proof}

\begin{remark}
The virtual control $\alpha$ selects the target manifold
\[
\mathcal M_\alpha := \{(x_1,x_2,z): x_2 = \alpha(x_1,z)\}.
\]
Because $e=0$ if and only if the state lies on $\mathcal M_\alpha$, the outer variable $e=x_2-\alpha(x_1,z)$ measures the deviation from that manifold. The Bolza functional $J_T^\alpha$ therefore has a direct geometric meaning. Its terminal term
\[
V(T)=V_1\bigl(x_1(T),z(T)\bigr)+\frac{1}{2}\|e(T)\|^2
\]
measures the residual first-step energy together with the residual distance from $\mathcal M_\alpha$ at the terminal time. Its running cost
\begin{align*}
\ell_\alpha(x_1,z,e,u)
&=\sigma\|\nabla f(x_1)\|^2 + \frac{c}{2}\|e\|^2 \\
&\quad + \frac{1}{2c}\|u-\dot\alpha+\lambda \nabla f(x_1)\|^2
\end{align*}
penalizes, respectively, the remaining gradient energy, the instantaneous deviation from the selected manifold, and the excess control effort relative to the exact backstepping correction needed to complete the square.
\end{remark}

From this viewpoint, the proposition says more than ``$u^\star$ stabilizes $e$'': once $\alpha$ has fixed the target manifold, the actual control law
\[
u^\star = \dot\alpha - ce - \lambda \nabla f(x_1)
\]
is precisely the control that steers the full system toward $\mathcal M_\alpha$ with minimum induced outer cost. In this sense, the inverse optimality of $u^\star$ is an outer tracking optimality relative to a fixed virtual-control design, and the functional $J_T^\alpha$ is not chosen ad hoc after the fact; it is the Bolza problem naturally generated by the storage identity associated with the chosen virtual control.
Figure~\ref{fig:single-layer-optimality-flow} summarizes this single-layer interpretation.

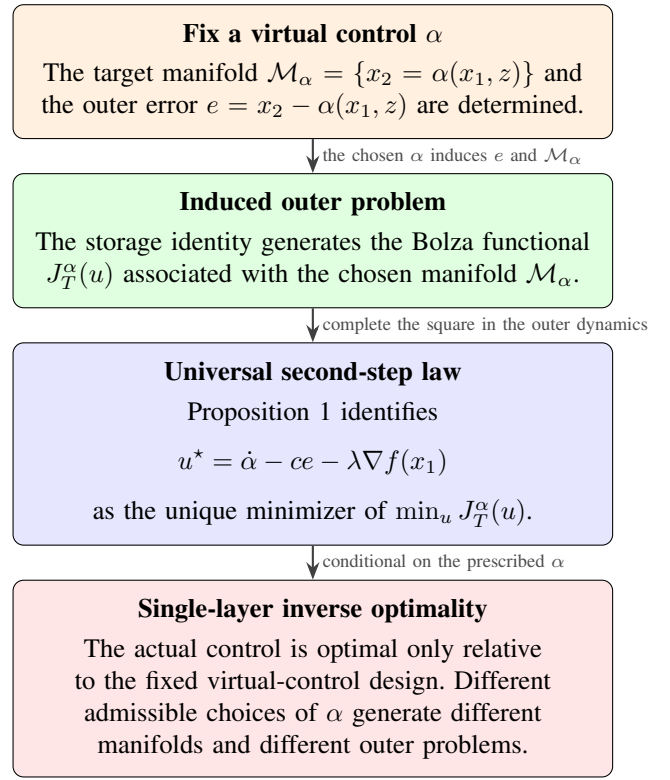
\begin{figure}[!t]
\centering
\begin{tikzpicture}[
  >=Stealth,
  box/.style={draw, rounded corners=4pt, align=center, text width=0.84\linewidth, minimum height=1.55cm, inner sep=7pt},
  arrow/.style={->, thick, draw=black!75},
  flowlabel/.style={font=\scriptsize, align=center, text=black!70}
]
\node[box, fill=orange!12] (alpha) {\textbf{Fix a virtual control $\alpha$}\\[3pt]
The target manifold $\mathcal M_\alpha=\{x_2=\alpha(x_1,z)\}$ and the outer error $e=x_2-\alpha(x_1,z)$ are determined.};

\node[box, fill=green!12, below=0.45cm of alpha] (cost) {\textbf{Induced outer problem}\\[3pt]
The storage identity generates the Bolza functional $J_T^\alpha(u)$ associated with the chosen manifold $\mathcal M_\alpha$.};

\node[box, fill=blue!10, below=0.45cm of cost] (u) {\textbf{Universal second-step law}\\[3pt]
Proposition~\ref{prop:fixed-virtual-optimality} identifies
\[
u^\star=\dot\alpha-ce-\lambda \nabla f(x_1)
\]
as the unique minimizer of $\min_u J_T^\alpha(u)$.};

\node[box, fill=red!10, below=0.45cm of u] (single) {\textbf{Single-layer inverse optimality}\\[3pt]
The actual control is optimal only relative to the fixed virtual-control design. Different admissible choices of $\alpha$ generate different manifolds and different outer problems.};

\draw[arrow] (alpha) -- node[flowlabel, right] {the chosen $\alpha$ induces $e$ and $\mathcal M_\alpha$} (cost);
\draw[arrow] (cost) -- node[flowlabel, right] {complete the square in the outer dynamics} (u);
\draw[arrow] (u) -- node[flowlabel, right] {conditional on the prescribed $\alpha$} (single);
\end{tikzpicture}
\caption{Single-layer optimal backstepping design for gradient flow. Once the virtual control $\alpha$ is fixed, it determines the target manifold and the induced outer Bolza functional. Proposition~\ref{prop:fixed-virtual-optimality} then yields the unique optimal actual control for this conditional outer problem, while Proposition~\ref{prop:virtual-design-freedom} explains why changing $\alpha$ changes the single-layer problem itself.}
\label{fig:single-layer-optimality-flow}
\end{figure}

\begin{remark}[Nesterov flow as an instance of Proposition~\ref{prop:fixed-virtual-optimality}]
The abstract statement of Proposition~\ref{prop:fixed-virtual-optimality} becomes very concrete for the gradient-type virtual control
\[
\alpha(x_1)=-k_1\nabla f(x_1),
\]
which later yields the constant-parameter Nesterov flow in Theorem~\ref{thm:nesterov-flow}. In this case,
$
V_1(x_1)=\phi(x_1)
$,
$
e=x_2+k_1\nabla f(x_1)
$,
$
\sigma=k_1
$,
$
\lambda=1
$,
and
$
c=k_2
$; the selected target manifold is
\[
\mathcal M_\alpha=\{(x_1,x_2): x_2=-k_1\nabla f(x_1)\}.
\]
The induced outer optimal control problem is therefore to minimize
\[
\begin{aligned}
J_{T}^{\mathrm{Nest}}(u)
:={}& \phi\bigl(x_1(T)\bigr) \\
&+ \frac{1}{2}\|x_2(T)+k_1\nabla f(x_1(T))\|^2 \\
&+ \int_0^T \Bigl[
k_1\|\nabla f(x_1)\|^2
+ \frac{k_2}{2}\|x_2+k_1\nabla f(x_1)\|^2 \\
&\qquad\qquad
+ \frac{1}{2k_2}\|u-\dot\alpha+\nabla f(x_1)\|^2
\Bigr]dt.
\end{aligned}
\]
Proposition~\ref{prop:fixed-virtual-optimality} then says that the Nesterov feedback term
\[
u^\star=-k_1\nabla^2f(x_1)x_2-k_2\bigl(x_2+k_1\nabla f(x_1)\bigr)-\nabla f(x_1)
\]
is exactly the unique minimizer of this problem. Thus, once the virtual control has declared that the desired first-step behavior is the gradient manifold $x_2=-k_1\nabla f(x_1)$, the Nesterov correction can be interpreted as the optimal way to drive the full second-order system toward that manifold while minimizing the induced outer cost.
\end{remark}

Proposition~\ref{prop:fixed-virtual-optimality} is therefore only a conditional statement: once $\alpha$ has been fixed, the second-step law is optimal for the induced outer problem. The next question is therefore which part of the backstepping construction still carries the real design freedom. Proposition~\ref{prop:virtual-design-freedom} answers this by showing that the universal second-step law is essentially fixed by the first-step dissipation identity, whereas the nontrivial freedom remains in the choice of the virtual control itself.

\begin{proposition}[Residual freedom in the virtual-control design]\label{prop:virtual-design-freedom}
Suppose that $V_1\in C^1$ and $\alpha\in C^1$ satisfy, for some constants $\sigma>0$ and $\lambda>0$, the first-step identity
\begin{equation}
\begin{aligned}
&\nabla_{x_1}V_1(x_1,z)^\top\bigl(\alpha(x_1,z)+e\bigr)  + \nabla_zV_1(x_1,z)^\top q(x_1,z) \\
&= -\sigma\|\nabla f(x_1)\|^2 + \lambda \nabla f(x_1)^\top e
\end{aligned}
\label{eq:first-step-identity}
\end{equation}
pointwise for all $(x_1,z,e)$. Then
\begin{equation}
\nabla_{x_1}V_1 = \lambda \nabla f(x_1),
\label{eq:value-gradient-match}
\end{equation}
\begin{equation}
\nabla_{x_1}V_1^\top\alpha + \nabla_zV_1^\top q
=
-\sigma\|\nabla f(x_1)\|^2
=
-\frac{\sigma}{\lambda^2}\|\nabla_{x_1}V_1\|^2.
\label{eq:reduced-dissipation}
\end{equation}
Moreover, on the set $\{(x_1,z): \nabla_{x_1}V_1\neq 0\}$, every admissible virtual control admits the decomposition
\begin{equation}
\begin{aligned}
\alpha
&= -\frac{\nabla_zV_1^\top q + \frac{\sigma}{\lambda^2}\|\nabla_{x_1}V_1\|^2}{\|\nabla_{x_1}V_1\|^2}\,\nabla_{x_1}V_1 + \beta, \\
&\qquad \nabla_{x_1}V_1^\top\beta = 0.
\end{aligned}
\label{eq:virtual-control-split}
\end{equation}
\end{proposition}

\begin{proof}
Fix $(x_1,z)$ and rewrite \eqref{eq:first-step-identity} as
\[
\begin{aligned}
\bigl(\nabla_{x_1}V_1-\lambda \nabla f(x_1)\bigr)^\top e
&=
-\sigma\|\nabla f(x_1)\|^2 \\
&\quad -\nabla_{x_1}V_1^\top\alpha-\nabla_zV_1^\top q.
\end{aligned}
\]
The right-hand side is independent of $e$. Since the identity holds for every $e$, we may first set $e=0$, which gives
\[
-\sigma\|\nabla f(x_1)\|^2-\nabla_{x_1}V_1^\top\alpha-\nabla_zV_1^\top q=0.
\]
Substituting this back into the previous relation yields
\[
\bigl(\nabla_{x_1}V_1-\lambda \nabla f(x_1)\bigr)^\top e = 0, \qquad \forall e\in\mathbb{R}^n.
\]
We now prove \eqref{eq:value-gradient-match} by contradiction. If $\nabla_{x_1}V_1-\lambda \nabla f(x_1)\neq 0$, choose
\[
e=\nabla_{x_1}V_1-\lambda \nabla f(x_1).
\]
Then
\[
\bigl(\nabla_{x_1}V_1-\lambda \nabla f(x_1)\bigr)^\top e
=
\|\nabla_{x_1}V_1-\lambda \nabla f(x_1)\|^2
>
0,
\]
This contradiction proves \eqref{eq:value-gradient-match}.

Substituting \eqref{eq:value-gradient-match} into the $e=0$ relation gives
\[
\nabla_{x_1}V_1^\top\alpha+\nabla_zV_1^\top q
=
-\sigma\|\nabla f(x_1)\|^2.
\]
Using $\nabla f(x_1)=\lambda^{-1}\nabla_{x_1}V_1$ then yields \eqref{eq:reduced-dissipation}.

Finally, on the set $\{\nabla_{x_1}V_1\neq 0\}$, decompose $\alpha$ as
\[
\alpha=a\,\nabla_{x_1}V_1+\beta,
\qquad
\nabla_{x_1}V_1^\top\beta=0.
\]
Substituting this into \eqref{eq:reduced-dissipation} gives
\[
a\|\nabla_{x_1}V_1\|^2+\nabla_zV_1^\top q
=
-\frac{\sigma}{\lambda^2}\|\nabla_{x_1}V_1\|^2,
\]
so
\[
a
=
-\frac{\nabla_zV_1^\top q+\frac{\sigma}{\lambda^2}\|\nabla_{x_1}V_1\|^2}{\|\nabla_{x_1}V_1\|^2}.
\]
Substituting this expression for $a$ back into $\alpha=a\,\nabla_{x_1}V_1+\beta$ gives \eqref{eq:virtual-control-split}.
\end{proof}

\begin{figure}[!t]
\centering
\begin{tikzpicture}[scale=0.95, line cap=round, line join=round, >=stealth]
  \coordinate (p) at (1.55,0.75);
  \coordinate (gradend) at (2.70,1.33);
  \coordinate (parend) at (0.63,0.29);
  \coordinate (betaend) at (0.97,1.90);
  \coordinate (alphaend) at (0.05,1.44);

  \shade[inner color=white, outer color=orange!22] (0,0) ellipse (2.35 and 1.45);
  \draw[gray!55] (0,0) ellipse (0.95 and 0.58);
  \draw[gray!60] (0,0) ellipse (1.60 and 0.98);
  \draw[gray!70, thick] (0,0) ellipse (2.25 and 1.35);
  \node[gray!75] at (-1.60,1.82) {$V_1=\mathrm{const.}$};
  \node[orange!80!black] at (1.95,1.82) {larger $V_1$};

  \fill (p) circle (1.2pt);
  \node[below right] at (p) {$x_1$};

  \draw[->, blue!80!black, thick] (p) -- (gradend)
    node[pos=1, right] {$\nabla_{x_1}V_1$};

  \draw[->, orange!90!black, thick] (p) -- (parend)
    node[pos=0.55, below] {$a\,\nabla_{x_1}V_1$};

  \draw[->, teal!70!black, thick] (p) -- (betaend)
    node[pos=0.65, above left] {$\beta$};

  \draw[black, very thick]
    (1.826,0.892) -- (1.687,1.168) -- (1.411,1.026);

  \draw[dashed, gray!70] (betaend) -- (alphaend);
  \draw[dashed, gray!70] (parend) -- (alphaend);

  \draw[->, red!80!black, very thick] (p) -- (alphaend)
    node[pos=0.75, above] {$\alpha$};

  \node[align=left, font=\scriptsize] at (-1.45,-1.75)
    {Normal direction fixed by dissipation\\Tangential direction free};
\end{tikzpicture}
\caption{Geometric interpretation of Proposition~\ref{prop:virtual-design-freedom}. The value of $V_1$ increases outward across its level sets, so $\nabla_{x_1}V_1$ points in the outward normal direction. The dissipation identity fixes the component of $\alpha$ parallel to $\nabla_{x_1}V_1$, while the orthogonal component $\beta$ remains free.}
\label{fig:virtual-control-split-geometry}
\end{figure}
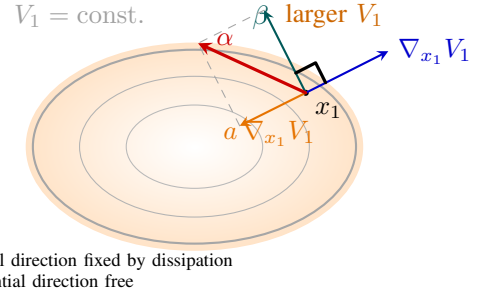

\begin{remark}
Equation~\eqref{eq:virtual-control-split} shows that, once $V_1$ is fixed, the dissipation requirement determines only the component of $\alpha$ parallel to $\nabla_{x_1}V_1$. The tangential component $\beta$ is free at first order. Figure~\ref{fig:virtual-control-split-geometry} visualizes this decomposition: $\nabla_{x_1}V_1$ is normal to a level set of $V_1$, the constrained part of $\alpha$ lies along this normal direction, and $\beta$ is the tangential residual freedom. Different choices of $\beta$ generate different manifolds $\mathcal M_\alpha$, different errors $e=x_2-\alpha$, different feedforward terms $\dot\alpha$, and therefore different induced performance indices $J_T^\alpha$. Hence the inverse optimality of $u^\star$ does not rank different backstepping designs; it is conditional on the chosen virtual control.
\end{remark}

A useful special case occurs when the auxiliary contribution vanishes, i.e. when $\nabla_zV_1^\top q=0$, and when $\beta=0$. Then \eqref{eq:virtual-control-split} reduces to
\[
\alpha = -\frac{\sigma}{\lambda^2}\nabla_{x_1}V_1 = -\frac{\sigma}{\lambda}\nabla f(x_1),
\]
which is a negative-gradient virtual control because $\lambda>0$. In particular, in the standard strict-feedback case with $V_1(x_1)=\phi(x_1)$, one has $\nabla_{x_1}V_1=\nabla f(x_1)$ and hence $\lambda=1$, so $\beta=0$ gives
\[
\alpha=-\sigma\nabla f(x_1).
\]
This is exactly the gradient-type virtual control used in Theorem~\ref{thm:nesterov-flow}, with $\sigma=k_1$.

Propositions~\ref{prop:fixed-virtual-optimality} and \ref{prop:virtual-design-freedom} therefore show that, once a virtual control $\alpha$ has been chosen, the universal second-step law is already optimally determined for the induced outer problem, whereas the genuine design freedom lies in the choice of $\alpha$ itself.
If one further requires the virtual control $\alpha$ to be optimal for a suitable reduced problem, then the entire backstepping construction acquires a genuine two-layer optimality interpretation.
Figure~\ref{fig:two-layer-optimality-flow} summarizes this hierarchy.

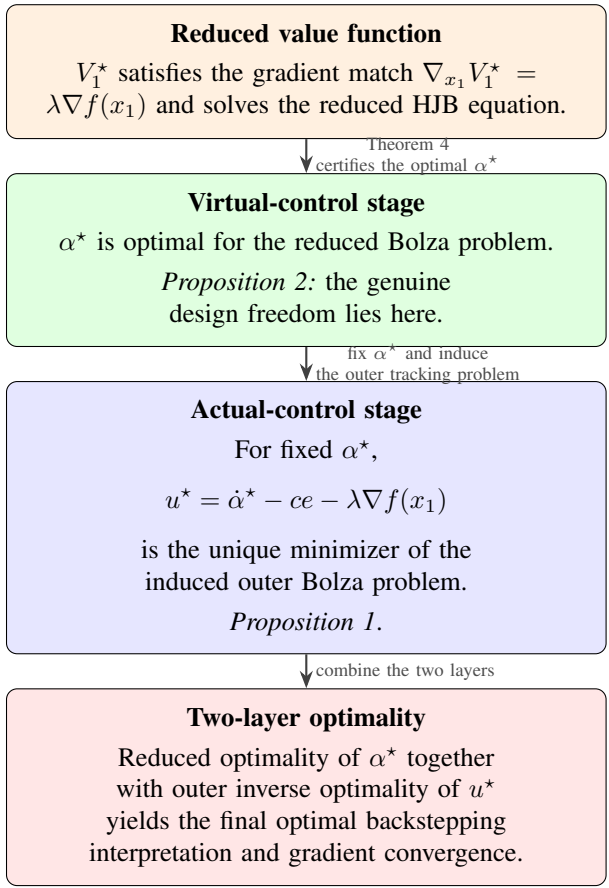
\begin{figure}[!t]
\centering
\begin{tikzpicture}[
  >=Stealth,
  box/.style={draw, rounded corners=4pt, align=center, text width=0.84\linewidth, minimum height=1.7cm, inner sep=7pt},
  arrow/.style={->, thick, draw=black!75},
  flowlabel/.style={font=\scriptsize, align=center, text=black!70}
]
\node[box, fill=orange!12] (v1) {\textbf{Reduced value function}\\[3pt]
$V_1^\star$ satisfies the gradient match
$\nabla_{x_1}V_1^\star=\lambda \nabla f(x_1)$
and solves the reduced HJB equation.};

\node[box, fill=green!12, below=0.45cm of v1] (alpha) {\textbf{Virtual-control stage}\\[3pt]
$\alpha^\star$ is optimal for the reduced Bolza problem.\\[3pt]
\emph{Proposition~\ref{prop:virtual-design-freedom}:} the genuine design freedom lies here.};

\node[box, fill=blue!10, below=0.45cm of alpha] (u) {\textbf{Actual-control stage}\\[3pt]
For fixed $\alpha^\star$,
\[
u^\star=\dot\alpha^\star-ce-\lambda \nabla f(x_1)
\]
is the unique minimizer of the induced outer Bolza problem.\\[3pt]
\emph{Proposition~\ref{prop:fixed-virtual-optimality}}.};

\node[box, fill=red!10, below=0.45cm of u] (flow) {\textbf{Two-layer optimality}\\[3pt]
Reduced optimality of $\alpha^\star$ together with outer inverse optimality of $u^\star$ yields the final optimal backstepping interpretation and gradient convergence.};

\draw[arrow] (v1) -- node[flowlabel, right] {Theorem~\ref{thm:optimal-backstepping}\\ certifies the optimal $\alpha^\star$} (alpha);
\draw[arrow] (alpha) -- node[flowlabel, right] {fix $\alpha^\star$ and induce\\ the outer tracking problem} (u);
\draw[arrow] (u) -- node[flowlabel, right] {combine the two layers} (flow);
\end{tikzpicture}
\caption{Two-layer optimal backstepping design for gradient flow. The reduced value function $V_1^\star$ does not by itself constitute an ``optimality'' claim; rather, by solving the reduced Hamilton--Jacobi--Bellman equation, it certifies that $\alpha^\star$ is optimal for the reduced problem. Once $\alpha^\star$ is fixed, Proposition~\ref{prop:fixed-virtual-optimality} identifies the optimal actual control for the induced outer problem, and the two layers combine into the optimal backstepping interpretation of Theorem~\ref{thm:optimal-backstepping}.}
\label{fig:two-layer-optimality-flow}
\end{figure}

\begin{theorem}[Optimal backstepping principle]\label{thm:optimal-backstepping}
Consider the augmented system \eqref{eq:augmented-strict-feedback} and a prescribed reduced running cost $L_r : \mathbb{R}^n \times \mathbb{R}^m \times \mathbb{R}^n \to \mathbb{R}_{\ge 0}$.
Assume that there exist a continuously differentiable reduced value function $V_1^\star(x_1,z)$, a continuously differentiable virtual control $\alpha^\star(x_1,z)$, and constants $c_\phi>0$, $c_z>0$, $\lambda>0$, and $\sigma>0$ such that
\begin{equation}
V_1^\star(x_1,z) \ge c_\phi\phi(x_1) + c_z\|z\|^2,
\label{eq:optimal-storage-bound}
\end{equation}
\begin{equation}
\nabla_{x_1}V_1^\star(x_1,z) = \lambda \nabla f(x_1),
\label{eq:optimal-gradient-match}
\end{equation}
and the stationary Hamilton--Jacobi--Bellman equation for the reduced system
\begin{equation}
\dot x_1 = a, \qquad \dot z = q(x_1,z),
\label{eq:reduced-system}
\end{equation}
namely
\begin{equation}
\begin{aligned}
0 = \min_{a\in\mathbb{R}^n}\Bigl\{&\nabla_{x_1}V_1^\star(x_1,z)^\top a \\
&+ \nabla_zV_1^\star(x_1,z)^\top q(x_1,z) + L_r(x_1,z,a)\Bigr\},
\end{aligned}
\label{eq:reduced-hjb}
\end{equation}
holds for all $(x_1,z)$, with minimizer $a=\alpha^\star(x_1,z)$. Assume in addition that
\begin{equation}
L_r\bigl(x_1,z,\alpha^\star(x_1,z)\bigr) \ge \sigma\|\nabla f(x_1)\|^2
\label{eq:reduced-cost-lower-bound}
\end{equation}
for all $(x_1,z)$, and that boundedness of $(\nabla f(x_1),z)$ implies boundedness of $\alpha^\star(x_1,z)$.

Define $e := x_2 - \alpha^\star(x_1,z)$, and choose the actual control
\begin{equation}
u^\star := \dot\alpha^\star - ce - \lambda \nabla f(x_1), \qquad c>0,
\label{eq:optimal-actual-control}
\end{equation}
where $\dot\alpha^\star$ denotes the total derivative of $\alpha^\star$ along \eqref{eq:augmented-strict-feedback}. Then the following statements hold.

\begin{enumerate}
\item For every $T>0$, $\alpha^\star$ minimizes the reduced Bolza functional
\[
\mathcal J_{r,T}(a)
:=
V_1^\star\bigl(x_1(T),z(T)\bigr)
+
\int_0^T L_r(x_1,z,a)\,dt
\]
over all admissible reduced controls $a(\cdot)$ for the system \eqref{eq:reduced-system}.

\item The full-order storage function
$
V^\star := V_1^\star + \frac{1}{2}\|e\|^2
$
satisfies
\begin{equation}
\dot V^\star = -L_r\bigl(x_1,z,\alpha^\star\bigr) - c\|e\|^2 \le -\sigma\|\nabla f(x_1)\|^2 - c\|e\|^2.
\label{eq:optimal-full-dissipation}
\end{equation}
Consequently,
\begin{equation}
\nabla f(x_1(t)) \to 0 \qquad \text{as } t\to\infty.
\label{eq:optimal-gradient-convergence}
\end{equation}

\item For fixed $\alpha^\star$, the actual control $u^\star$ is the unique minimizer, on every finite horizon $[0,T]$, of the outer Bolza functional
\[
\begin{aligned}
\mathcal J_{o,T}^{\alpha^\star}(u)
&:= V^\star\bigl(x_1(T),z(T),e(T)\bigr) \\
&\quad + \int_0^T \Bigl[
L_r\bigl(x_1,z,\alpha^\star\bigr) \\
&\qquad + \frac{c}{2}\|e\|^2
+ \frac{1}{2c}\|u-\dot\alpha^\star+\lambda \nabla f(x_1)\|^2
\Bigr]dt.
\end{aligned}
\]
Therefore the resulting design has a two-layer interpretation: $\alpha^\star$ is optimal for a prescribed reduced problem, whereas $u^\star$ is inverse optimal for the induced outer tracking problem.
\end{enumerate}
\end{theorem}

\begin{proof}
By \eqref{eq:reduced-hjb}, for every admissible reduced control $a$ one has
\[
\nabla_{x_1}V_1^\star{}^\top a + \nabla_zV_1^\star{}^\top q + L_r(x_1,z,a) \ge 0,
\]
with equality at $a=\alpha^\star$. Along the reduced system \eqref{eq:reduced-system}, this yields $\dot V_1^\star + L_r(x_1,z,a) \ge 0$. Integration on $[0,T]$ therefore gives
\[
\mathcal J_{r,T}(a) \ge V_1^\star\bigl(x_1(0),z(0)\bigr),
\]
with equality for $a=\alpha^\star$, which proves item 1.

Next, along the full-order system with the control \eqref{eq:optimal-actual-control}, one has
\[
\begin{aligned}
\dot V_1^\star
&=
\nabla_{x_1}V_1^\star{}^\top\bigl(\alpha^\star+e\bigr)
+ \nabla_zV_1^\star{}^\top q \\
&=
- L_r\bigl(x_1,z,\alpha^\star\bigr) + \lambda \nabla f(x_1)^\top e,
\end{aligned}
\]
where the last identity follows from \eqref{eq:reduced-hjb} evaluated at the minimizing virtual control $a=\alpha^\star$ together with \eqref{eq:optimal-gradient-match}. Since $\dot e = u^\star - \dot\alpha^\star = -ce-\lambda \nabla f(x_1)$, it follows that
\[
\dot V^\star = \dot V_1^\star + e^\top\dot e = -L_r\bigl(x_1,z,\alpha^\star\bigr) - c\|e\|^2,
\]
which proves the equality in \eqref{eq:optimal-full-dissipation}. By \eqref{eq:optimal-storage-bound}, $V^\star(t)$ is bounded from below and nonincreasing, hence $\phi(x_1(t))$, $z(t)$, and $e(t)$ are bounded. Lemma~\ref{lem:gradient-suboptimality} yields boundedness of $\nabla f(x_1(t))$, and the assumed boundedness implication then gives boundedness of $\alpha^\star(x_1(t),z(t))$. Therefore $x_2(t)=e(t)+\alpha^\star(x_1(t),z(t))$ is bounded. Moreover, integrating \eqref{eq:optimal-full-dissipation} and using \eqref{eq:reduced-cost-lower-bound} gives
\[
\int_0^\infty \|\nabla f(x_1(t))\|^2\,dt \le \frac{V^\star(0)-\inf V^\star}{\sigma} < \infty.
\]
Lemma~\ref{lem:gradient-vanishing} therefore yields \eqref{eq:optimal-gradient-convergence}, which proves item 2.

Finally, fix $\alpha^\star$ and let $u$ be arbitrary. Then
\[
\dot V^\star
=
- L_r\bigl(x_1,z,\alpha^\star\bigr)
+ e^\top\bigl(u-\dot\alpha^\star+\lambda \nabla f(x_1)\bigr).
\]
Hence
\[
\begin{aligned}
\dot V^\star
&+ L_r\bigl(x_1,z,\alpha^\star\bigr)
+ \frac{c}{2}\|e\|^2 \\
&+ \frac{1}{2c}\|u-\dot\alpha^\star+\lambda \nabla f(x_1)\|^2 \\
&= \frac{1}{2c}\|u-\dot\alpha^\star+\lambda \nabla f(x_1) + ce\|^2 \ge 0.
\end{aligned}
\]
Integration on $[0,T]$ yields $\mathcal J_{o,T}^{\alpha^\star}(u) \ge V^\star(0)$, and equality holds if and only if $u=u^\star$. This proves item 3.
\end{proof}

\section{Examples}\label{sec:examples}

This section illustrates how Theorem~\ref{thm:optimal-backstepping} generates concrete optimization dynamics from the two-layer optimality principle. We begin with a fully explicit quadratic benchmark, then show how strong convexity sharpens the conclusion to exponential convergence, and finally return to a more general augmented example that works for an arbitrary convex $L$-smooth objective without prescribing the detailed form of $f$.

\subsection*{Case 1: Quadratic Objective}

Theorem~\ref{thm:optimal-backstepping} admits a particularly transparent test example in the standard strict-feedback case with no auxiliary state and quadratic objective
\[
f(x_1)=\frac{1}{2}\|x_1\|^2,
\qquad
f^\star=0,
\qquad
\nabla f(x_1)=x_1.
\]
Fix a constant $k>0$ and consider the reduced running cost
\[
L_r(x_1,a):=\frac{1}{2k}\|a\|^2+\frac{k}{2}\|x_1\|^2.
\]
Choose
\[
V_1^\star(x_1)=\phi(x_1)=\frac{1}{2}\|x_1\|^2.
\]
Then $\nabla_{x_1}V_1^\star=\nabla f(x_1)$, so $\lambda=1$, and the reduced Hamilton--Jacobi--Bellman equation becomes
\[
0=\min_{a\in\mathbb{R}^n}
\left\{
x_1^\top a+\frac{1}{2k}\|a\|^2+\frac{k}{2}\|x_1\|^2
\right\}.
\]
Its unique minimizer is
\[
\alpha^\star(x_1)=-k x_1=-k\nabla f(x_1).
\]
Hence the first layer of optimality is explicit: $\alpha^\star$ minimizes the reduced Bolza functional
\[
\mathcal J_{r,T}(a)
=
\frac{1}{2}\|x_1(T)\|^2
+
\int_0^T\left(
\frac{1}{2k}\|a\|^2+\frac{k}{2}\|x_1\|^2
\right)dt.
\]

Now define the backstepping error
\[
e=x_2-\alpha^\star(x_1)=x_2+kx_1.
\]
Since $\dot\alpha^\star=-k x_2$, Theorem~\ref{thm:optimal-backstepping} gives the actual control
\[
\begin{aligned}
u^\star
&=\dot\alpha^\star-ce-\nabla f(x_1) \\
&=-k x_2-c(x_2+kx_1)-x_1 \\
&=-(k+c)x_2-(1+ck)x_1.
\end{aligned}
\]
Therefore the final optimization dynamics are
\[
\dot x_1=x_2,
\qquad
\dot x_2=-(k+c)x_2-(1+ck)x_1.
\]
This algorithm is genuinely produced by the two-layer principle: the virtual control $\alpha^\star=-k\nabla f(x_1)$ is optimal for the reduced problem, and then the actual control $u^\star$ is inverse optimal for the induced outer problem. Indeed,
\[
L_r\bigl(x_1,\alpha^\star(x_1)\bigr)=k\|x_1\|^2=k\|\nabla f(x_1)\|^2,
\]
so Theorem~\ref{thm:optimal-backstepping} yields the Lyapunov identity
\[
\begin{aligned}
V^\star
&=
\frac{1}{2}\|x_1\|^2+\frac{1}{2}\|x_2+kx_1\|^2, \\
\dot V^\star
&=
-k\|x_1\|^2-c\|x_2+kx_1\|^2.
\end{aligned}
\]
Thus this quadratic system is a clean benchmark for verifying the two-layer optimality interpretation. In the scalar case with $k=1$ and $c=1$, one obtains the explicit algorithm
\[
\dot x_1=x_2,
\qquad
\dot x_2=-2x_2-2x_1.
\]

\subsection*{Case 2: Strongly Convex Objectives}

Suppose now that $f$ is in addition $m$-strongly convex with $m>0$. Then
\[
\|\nabla f(x_1)\|^2 \ge 2m\phi(x_1),
\]
so the two-layer construction can be specialized to yield an exponentially stable optimization flow. Consider the standard strict-feedback case without auxiliary state, and choose the reduced running cost
\[
\begin{aligned}
L_r(x_1,a)
&:=
\frac{1}{2k}\|a\|^2+\frac{k}{2}\|\nabla f(x_1)\|^2+\mu\phi(x_1), \\
&\qquad k>0,\ \mu>0.
\end{aligned}
\]
Take
\[
V_1^\star(x_1)=\phi(x_1).
\]
Then $\nabla_{x_1}V_1^\star=\nabla f(x_1)$, so $\lambda=1$, and the reduced Hamilton--Jacobi--Bellman equation becomes
\[
\begin{aligned}
0=\min_{a\in\mathbb{R}^n}
\Bigl\{
&\nabla f(x_1)^\top a+\frac{1}{2k}\|a\|^2 \\
&+\frac{k}{2}\|\nabla f(x_1)\|^2+\mu\phi(x_1)
\Bigr\}.
\end{aligned}
\]
Its unique minimizer is
\[
\alpha^\star(x_1)=-k\nabla f(x_1).
\]
Hence the first layer of optimality is again explicit: $\alpha^\star$ is optimal for the reduced Bolza problem induced by $L_r$.

Defining
\[
e=x_2-\alpha^\star(x_1)=x_2+k\nabla f(x_1),
\]
Theorem~\ref{thm:optimal-backstepping} yields the actual control
\[
\begin{aligned}
u^\star
&=\dot\alpha^\star-ce-\nabla f(x_1) \\
&=-k\nabla^2 f(x_1)x_2-c\bigl(x_2+k\nabla f(x_1)\bigr)-\nabla f(x_1) \\
&=-\bigl(k\nabla^2 f(x_1)+cI\bigr)x_2-(1+ck)\nabla f(x_1).
\end{aligned}
\]
Therefore the resulting two-layer optimality algorithm is
\[
\dot x_1=x_2,
\qquad
\dot x_2=-\bigl(k\nabla^2 f(x_1)+cI\bigr)x_2-(1+ck)\nabla f(x_1).
\]
It is generated by the same two-layer mechanism as before: $\alpha^\star=-k\nabla f(x_1)$ is reduced-optimal, and $u^\star$ is outer inverse-optimal for the induced tracking problem.

Moreover,
\[
L_r\bigl(x_1,\alpha^\star(x_1)\bigr)=k\|\nabla f(x_1)\|^2+\mu\phi(x_1),
\]
so Theorem~\ref{thm:optimal-backstepping} gives
\[
\begin{aligned}
\dot V^\star&=-k\|\nabla f(x_1)\|^2-\mu\phi(x_1)-c\|e\|^2, \\
V^\star&=\phi(x_1)+\frac{1}{2}\|e\|^2.
\end{aligned}
\]
Using strong convexity,
\[
\dot V^\star
\le
-(2mk+\mu)\phi(x_1)-c\|e\|^2
\le
-\min\{2mk+\mu,\,2c\}V^\star.
\]
Hence $V^\star$ decays exponentially, and therefore both $\phi(x_1(t))$ and the tracking error $e(t)=x_2(t)+k\nabla f(x_1(t))$ converge to zero exponentially fast. This provides a clean strong-convexity example in which the framework yields not only a two-layer optimality interpretation, but also an explicit exponential convergence estimate.

\subsection*{Case 3: General Convex Smooth Objectives}

The previous two cases either fix the objective explicitly or add strong convexity. To test the genuine scope of Theorem~\ref{thm:optimal-backstepping}, it is more informative to consider an example that works for an arbitrary objective satisfying Assumption~\ref{asm:convex-smooth-objective}. Introduce the auxiliary state
\[
\dot z=-\rho z+\nabla f(x_1),
\qquad
\rho>0,
\]
and consider the reduced system
\[
\dot x_1=a,
\qquad
\dot z=-\rho z+\nabla f(x_1),
\]
where $a$ is the virtual control. Fix constants $k_d>0$ and $k_i>0$, and choose
\[
V_1^\star(x_1,z):=\phi(x_1)+\frac{k_i}{2}\|z\|^2.
\]
Then
\[
\nabla_{x_1}V_1^\star=\nabla f(x_1),
\qquad
\nabla_zV_1^\star=k_i z,
\]
so the gradient-matching condition of Theorem~\ref{thm:optimal-backstepping} holds with $\lambda=1$. Now define the reduced running cost
\[
L_r(x_1,z,a)
:=
\frac{1}{2k_d}\|a+k_i z\|^2
+
\frac{k_d}{2}\|\nabla f(x_1)\|^2
+
k_i\rho\|z\|^2.
\]
Then the reduced Hamilton--Jacobi--Bellman expression becomes
\[
\begin{aligned}
&\nabla_{x_1}V_1^\star{}^\top a
+
\nabla_zV_1^\star{}^\top\bigl(-\rho z+\nabla f(x_1)\bigr)
+
L_r(x_1,z,a) \\
&=
\nabla f(x_1)^\top(a+k_i z)
+
\frac{1}{2k_d}\|a+k_i z\|^2
+
\frac{k_d}{2}\|\nabla f(x_1)\|^2.
\end{aligned}
\]
Hence the unique minimizer is
\[
\alpha^\star(x_1,z)=-k_d\nabla f(x_1)-k_i z.
\]
Therefore the first layer of optimality is completely explicit even though the objective $f$ itself has not been specified beyond convexity and smoothness.

Now define
\[
e=x_2-\alpha^\star(x_1,z)=x_2+k_d\nabla f(x_1)+k_i z.
\]
Since
\[
\begin{aligned}
\dot\alpha^\star
&=-k_d\nabla^2 f(x_1)x_2-k_i\dot z \\
&=-k_d\nabla^2 f(x_1)x_2+k_i\rho z-k_i\nabla f(x_1),
\end{aligned}
\]
Theorem~\ref{thm:optimal-backstepping} yields the actual control
\[
\begin{aligned}
u^\star
&=\dot\alpha^\star-ce-\nabla f(x_1) \\
&=-k_d\nabla^2 f(x_1)x_2-cx_2 \\
&\quad -(1+k_i+ck_d)\nabla f(x_1)+k_i(\rho-c)z.
\end{aligned}
\]
The resulting optimization dynamics are
\[
\begin{aligned}
\dot x_1&=x_2, \\
\dot x_2&=-k_d\nabla^2 f(x_1)x_2-cx_2 \\
&\quad -(1+k_i+ck_d)\nabla f(x_1)+k_i(\rho-c)z, \\
\dot z&=-\rho z+\nabla f(x_1).
\end{aligned}
\]
This is a genuinely two-layer construction for an arbitrary convex $L$-smooth objective. The virtual control
\[
\alpha^\star=-k_d\nabla f(x_1)-k_i z
\]
is optimal for the reduced problem generated by $L_r$, and then the actual control $u^\star$ is inverse optimal for the induced outer problem. Moreover,
\[
\begin{aligned}
L_r\bigl(x_1,z,\alpha^\star(x_1,z)\bigr)
&=
k_d\|\nabla f(x_1)\|^2+k_i\rho\|z\|^2 \\
&\ge
k_d\|\nabla f(x_1)\|^2,
\end{aligned}
\]
so Theorem~\ref{thm:optimal-backstepping} gives the dissipation identity
\[
\begin{aligned}
\dot V^\star
&=
-k_d\|\nabla f(x_1)\|^2-k_i\rho\|z\|^2-c\|e\|^2, \\
V^\star
&=
\phi(x_1)+\frac{k_i}{2}\|z\|^2+\frac{1}{2}\|e\|^2.
\end{aligned}
\]
In this sense, the example is both more complex and more general than the quadratic benchmark: it does not presuppose a particular objective function, yet it still yields an explicit algorithm from the two-layer optimality principle.

\section{Numerical Experiments}
\label{sec:experiments}

To validate the theoretical findings and evaluate the performance of the proposed backstepping-based optimization algorithms, we conduct numerical simulations for both strongly convex and general convex objective functions. For each scenario, we select five representative objective functions and solve the corresponding continuous-time gradient flows derived in Section V. To illustrate the convergence behavior, we plot the trajectory of the gradient norm $\|\nabla f(x_1(t))\|$ over time for each test case.

\subsection{Experiments for Case 2: Strongly Convex Objectives}

For the strongly convex setting, the backstepping framework yields the exponentially stable second-order dynamics given by:
\begin{align*}
    \dot{x}_1 = x_2, \quad \dot{x}_2 = -(k\nabla^2 f(x_1) + cI)x_2 - (1+ck)\nabla f(x_1).
\end{align*}
To demonstrate the robustness of this algorithm, we select the following five $L$-smooth and $m$-strongly convex objective functions $f: \mathbb{R}^2 \to \mathbb{R}$:
\begin{itemize}
    \item \textbf{F1 (Anisotropic Quadratic):} $f(x) = \frac{1}{2}x_1^2 + 5x_2^2$. A standard benchmark for linear convergence rates.
    \item \textbf{F2 (Shifted Quadratic):} $f(x) = (x_1 - 2)^2 + (x_2 + 2)^2$. Tests convergence to a non-origin minimizer.
    \item \textbf{F3 (Regularized Quartic):} $f(x) = \frac{1}{2}(x_1^2 + x_2^2) + \frac{1}{10}(x_1^4 + x_2^4)$. Introduces polynomial nonlinearity while maintaining strong convexity.
    \item \textbf{F4 (Hyperbolic Cosine):} $f(x) = \cosh(x_1) + \cosh(x_2) + \frac{1}{2}(x_1^2 + x_2^2)$. Features exponential growth in the gradient.
    \item \textbf{F5 (Regularized Pseudo-Huber):} $f(x) = \sqrt{1+x_1^2} + \sqrt{1+x_2^2} + \frac{1}{10}(x_1^2 + x_2^2)$. A smooth approximation of the $L_1$ norm with an added $L_2$ penalty to strictly satisfy strong convexity.
\end{itemize}

\begin{figure}[htbp]
    \centering
    \includegraphics[width=\linewidth]{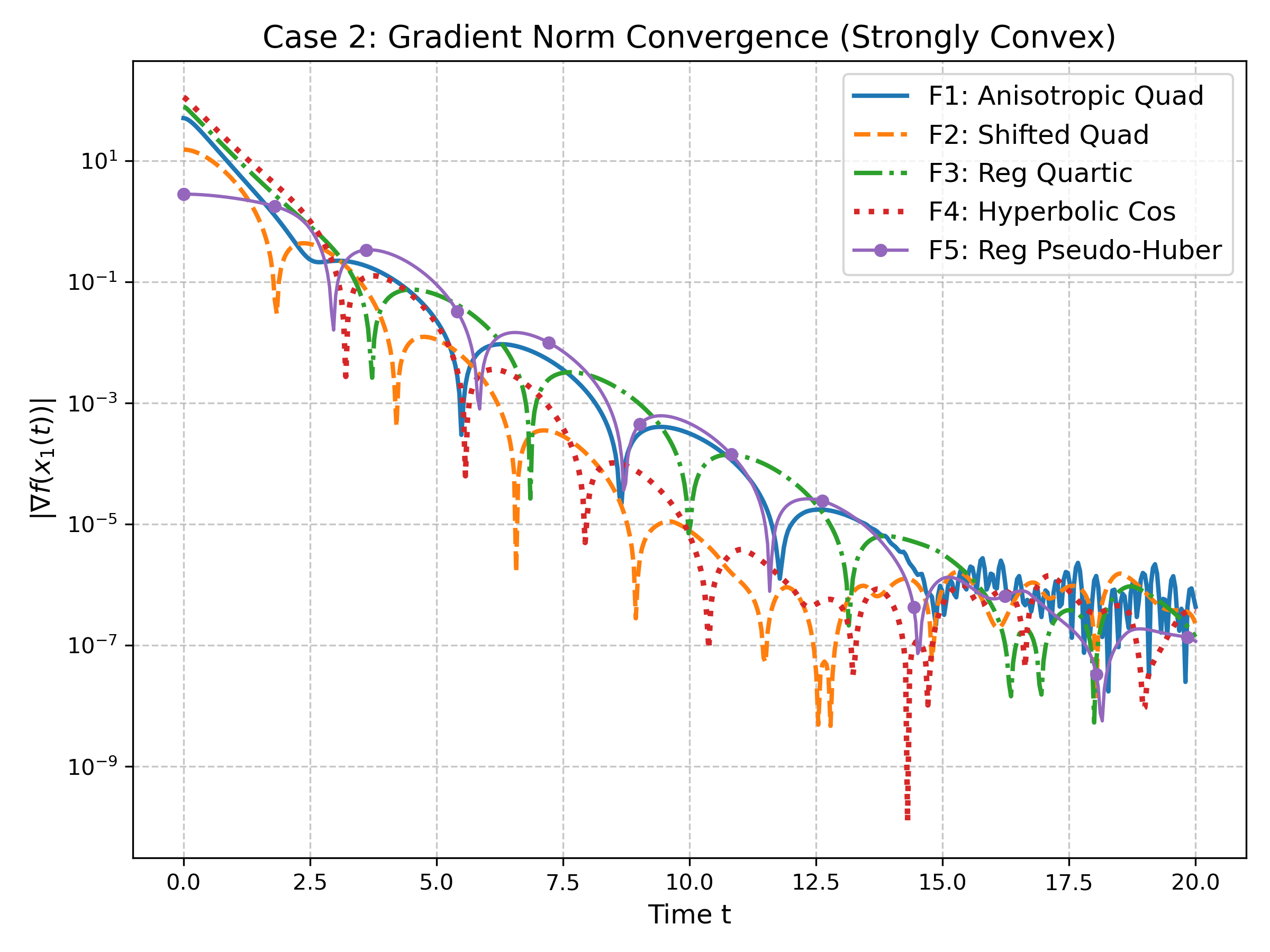}
    \caption{Evolution of the gradient norm $\|\nabla f(x_1(t))\|$ over time for five representative strongly convex objective functions (Case 2). The logarithmic scale illustrates the exponential decay (linear convergence rate) achieved by the proposed strict-feedback backstepping dynamics.}
    \label{fig:case2}
\end{figure}

In our simulations, we set the design parameters to $k = 1.0$ and $c = 1.0$, and initialize the states at $x_1(0) = [5, 5]^\top$ and $x_2(0) = [0, 0]^\top$. The experimental results, displaying the evolution of the gradient norm, are presented in Figure \ref{fig:case2}. Consistent with theoretical expectations, all five functions exhibit a linear convergence rate (exponential decay in time), successfully driving the gradient norm to zero.

\subsection{Experiments for Case 3: General Convex Smooth Objectives}

For the general convex and smooth setting, we evaluate the augmented backstepping dynamics involving the auxiliary state $z$:
\begin{align*}
    \dot{x}_1 =& x_2,\\
    \dot{x}_2 =& -k_d \nabla^2 f(x_1) x_2 - c x_2 - (1+k_i+c k_d)\nabla f(x_1) \\
    &+ k_i(\rho - c)z, \\
    \dot{z} =& -\rho z + \nabla f(x_1).
\end{align*}
We deliberately select five functions that lack strong convexity or possess degenerate directions:
\begin{itemize}
    \item \textbf{G1 (Pure Pseudo-Huber):} $f(x) = \sqrt{1+x_1^2+x_2^2} - 1$. The Hessian diminishes for large $\|x\|$, lacking global strong convexity.
    \item \textbf{G2 (Softplus/Logistic Loss):} $f(x) = \ln(e^{x_1} + e^{-x_1}) + \ln(e^{x_2} + e^{-x_2})$. Flat regions in the tails slow down standard gradient methods.
    \item \textbf{G3 (Pure Quartic):} $f(x) = \frac{1}{4}(x_1^4 + x_2^4)$. The Hessian vanishes exactly at the optimum $x^* = [0,0]^\top$.
    \item \textbf{G4 (Degenerate Quadratic):} $f(x) = \frac{1}{2}(x_1 - x_2)^2$. The minimizer is not a single point but a line $x_1 = x_2$, representing an ill-conditioned scenario.
    \item \textbf{G5 (Coupled Pseudo-Huber):} $f(x) = \sqrt{1 + (x_1+x_2)^2} - 1$. Introduces coupling between the variables without strict convexity.
\end{itemize}


\begin{figure}[htbp]
    \centering
    \includegraphics[width=\linewidth]{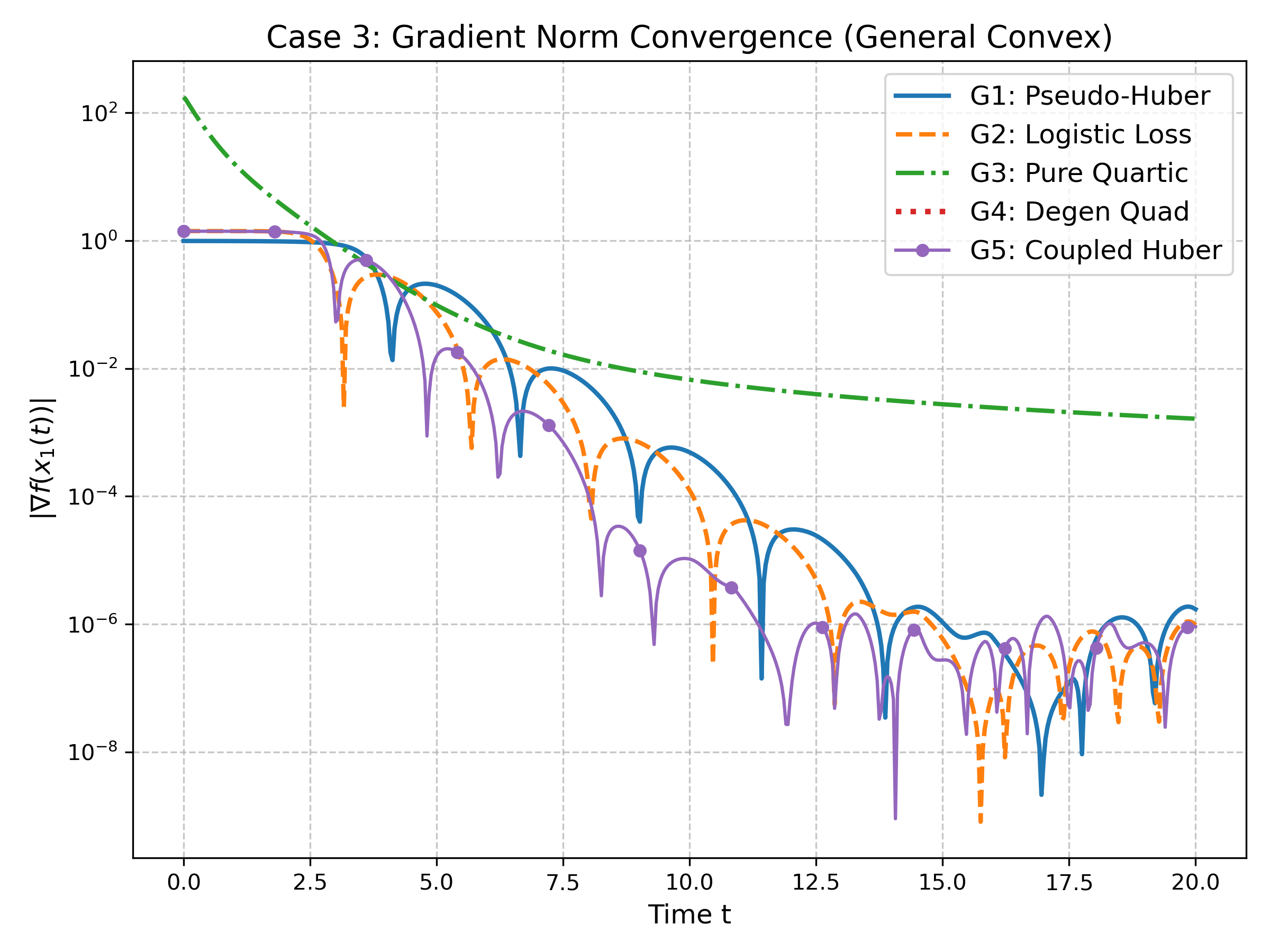}
    \caption{Evolution of the gradient norm $\|\nabla f(x_1(t))\|$ over time for five general convex and smooth objective functions (Case 3). The results demonstrate the successful asymptotic convergence to stationarity of the augmented two-layer optimal backstepping framework, even in the absence of global strong convexity or in the presence of degenerate directions.}
    \label{fig:case3}
\end{figure}

We set the control gains to $k_d = 1.0$, $k_i = 1.0$, $c = 1.0$, and $\rho = 2.0$. Initial conditions are chosen as $x_1(0) = [5, 5]^\top$, $x_2(0) = [0, 0]^\top$, and $z(0) = [0, 0]^\top$. As shown in Figure \ref{fig:case3}, the proposed two-layer optimal backstepping framework successfully drives the gradient norm to zero across all examples. Even in cases with vanishing Hessians (G3) or non-unique optima (G4), the algorithm remains stable and ensures asymptotic convergence to stationarity.

\section{Conclusion}\label{sec:conclusion}

This paper developed a backstepping-based framework for designing continuous-time unconstrained accelerated optimization algorithms. The starting point was to reformulate the algorithm-design problem as a controller-synthesis problem for the general augmented strict-feedback system \eqref{eq:augmented-strict-feedback}, with the gradient output $\nabla f(x_1)$, rather than the state $x_1$ itself, as the regulated variable. From this viewpoint, backstepping leads naturally to a two-stage design procedure: one first selects a virtual control $\alpha$ together with a first-step storage function $V_1$, and then constructs the actual input $u$ so that the full system is driven toward the manifold selected by $\alpha$ and the output $\nabla f(x_1(t))$ converges to zero.

This synthesis viewpoint is broad enough to recover existing accelerated optimization flows. In particular, the constant-parameter Nesterov flow and the PIDAO flow arise in this paper as direct theorem-level realizations obtained from different choices of virtual control and first-step dissipation structure. Beyond synthesis, the paper also established an optimal-backstepping interpretation of these constructions. It was shown that, once a virtual control $\alpha$ is fixed, the universal second-step law is inverse optimal only for the induced outer-tracking problem associated with the manifold determined by $\alpha$. The paper then formulated a genuine optimal-backstepping principle by requiring the virtual control itself to solve a prescribed reduced Hamilton--Jacobi--Bellman problem. Under this condition, the final design acquires a true two-layer optimality structure: the virtual control is optimal for the reduced problem, while the actual control is inverse optimal for the induced outer problem.

The framework also suggests several natural extensions. First, one may replace the Euclidean energy structure used here by a Bregman-type geometry, thereby connecting backstepping design with non-Euclidean accelerated flows. Second, the same viewpoint may be extended from unconstrained optimization to constrained optimization by treating optimization dynamics as backstepping-based feedback regulation laws compatible with feasibility constraints, projected dynamics, or primal--dual augmentations. Third, the framework opens a path toward online and time-varying optimization, where the objective and its minimizer evolve with time, so that the relevant goal is no longer convergence to a static stationary point, but feedback tracking of a moving optimum. These directions indicate that backstepping may provide not only a synthesis tool for unconstrained accelerated gradient flows, but also a broader control-theoretic design principle for structured optimization dynamics.

\appendices

\section*{References}
\bibliographystyle{unsrt}
\bibliography{reference} 

\vspace{-4 em}
\begin{IEEEbiography}[{\includegraphics[width=1in,height=1.25in,clip,keepaspectratio]{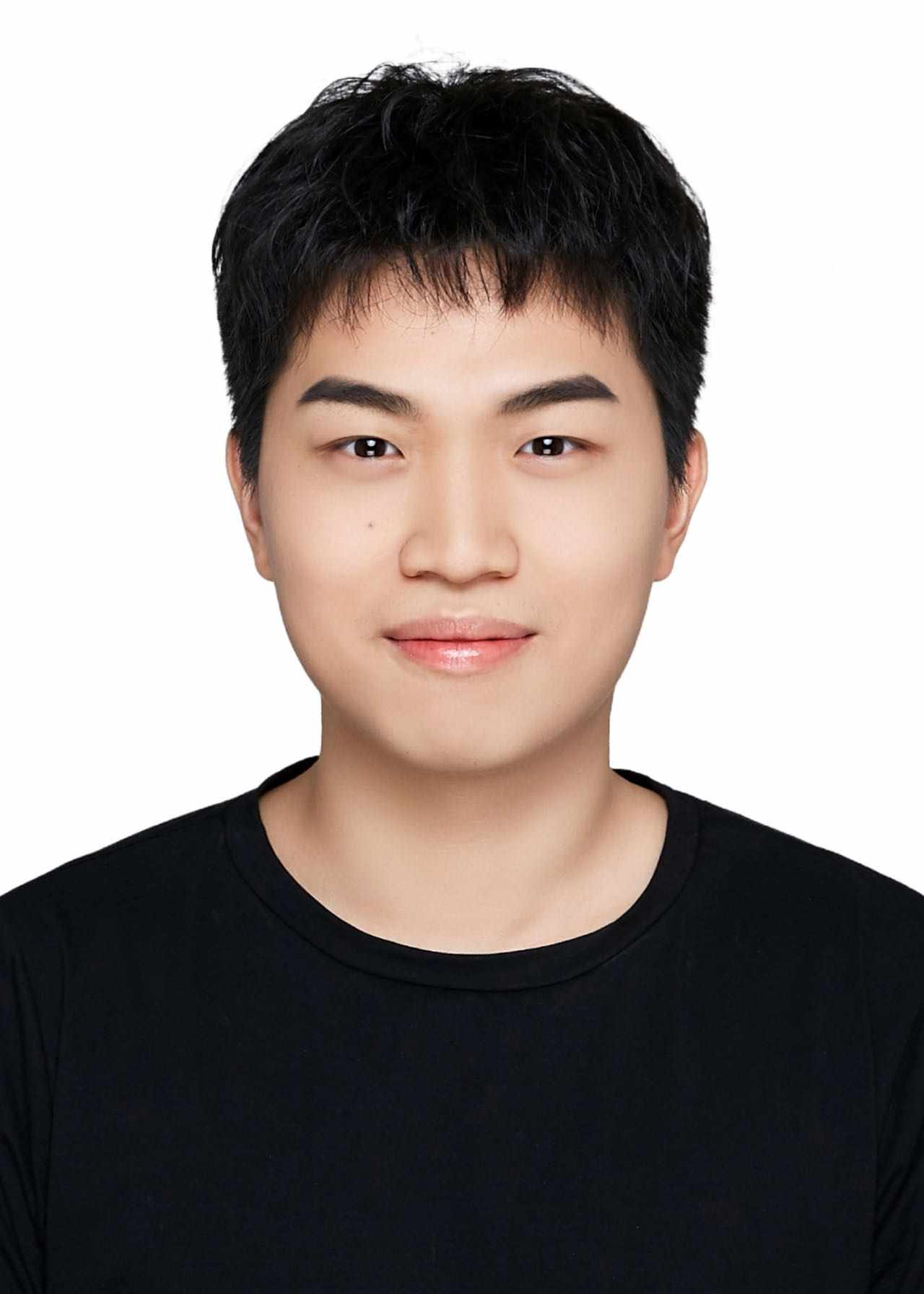}}]
{Song Chen} received  the Ph.D. degree in operational research and cybernetics from Zhejiang University, Hangzhou, China, in 2025.

He is currently a Research Fellow with the Department of Mathematics, National University of Singapore (NUS), Singapore. His research interests lie at the intersection of control theory and artificial intelligence, with a particular focus on control-oriented learning methods and embodied AI. His broader expertise includes convex optimization, nonlinear control, and machine learning theory with applications in robotics.
\end{IEEEbiography}

\vspace{-4 em}
\begin{IEEEbiography}[{\includegraphics[width=1in,height=1.25in,clip,keepaspectratio]{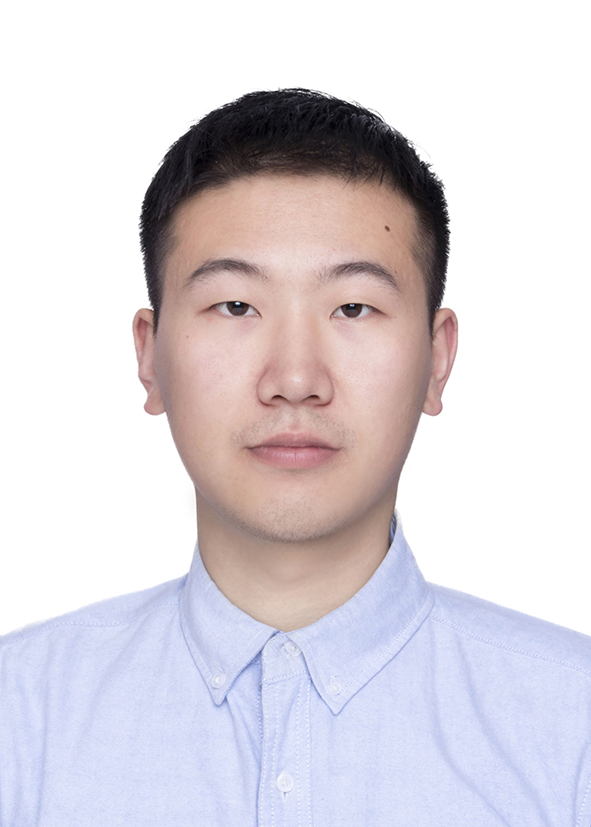}}]
{Jiaxu Liu} received the bachelor’s degree in mathematics from Renmin University of China, Beijing, China, in 2021. He is currently working toward the Ph.D. degree in operational research and cybernetics with Zhejiang University, Hangzhou, China.

His research interests include distributed optimization, convex optimization, robust control,  machine learning theory, and their applications in robotics.
\end{IEEEbiography}
\vspace{-4 em}

\begin{IEEEbiography}[{\includegraphics[width=1in,height=1.25in,clip,keepaspectratio]{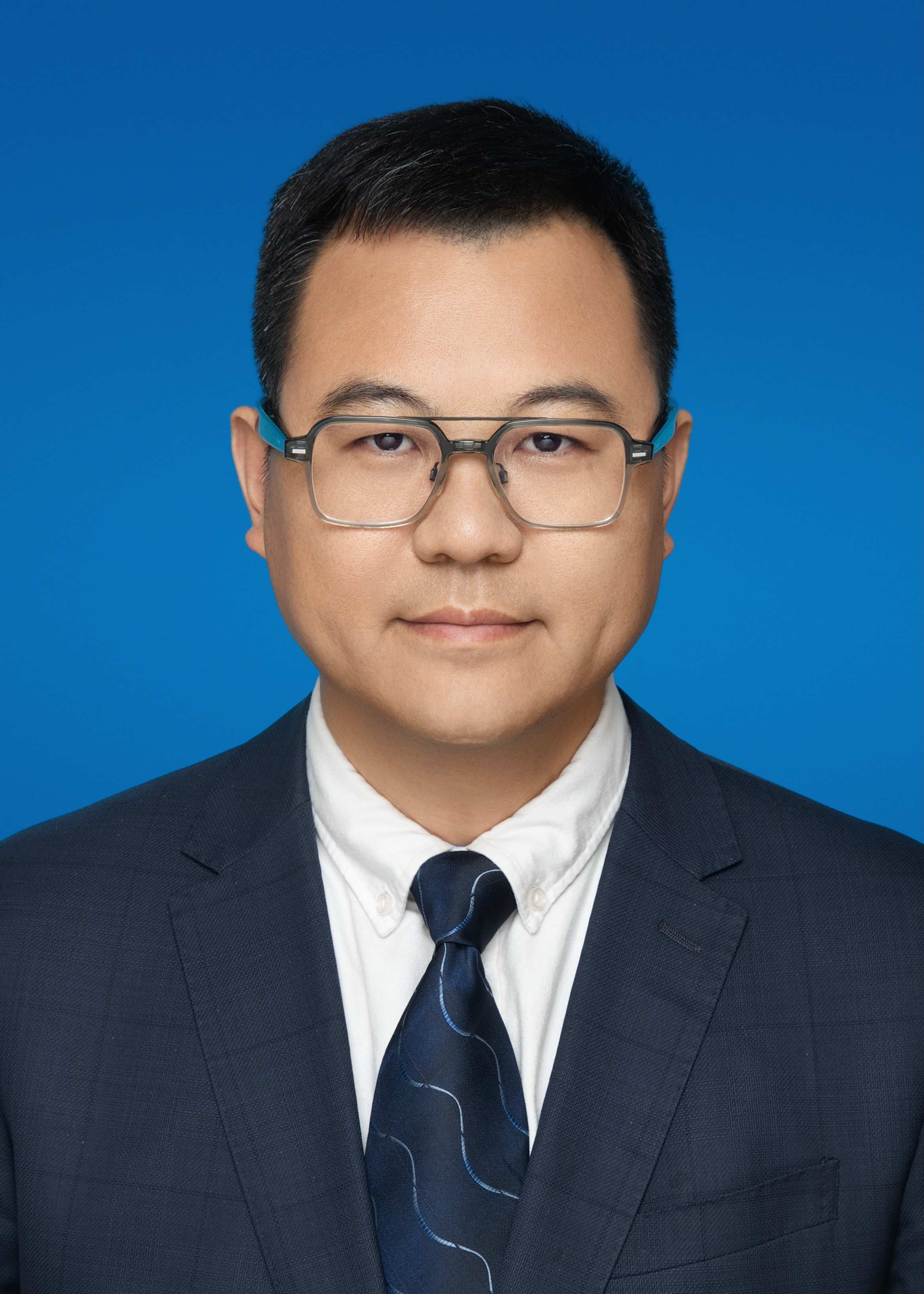}}]
{Chao Xu} (Senior Member, IEEE), received the Ph.D. degree in mechanical engineering from Lehigh University, Bethlehem, PA, USA, in 2010.

He is currently Professor of Controls and Autonomous Systems with the College of Control Science $\&$ Engineering, Zhejiang University (ZJU). He serves the inaugural Dean of ZJU Huzhou Institute, as well as plays the role of the Managing Editor for two international journals, e.g., \textit{IET Cyber-Systems and  Robotics} (IET-CSR), and \textit{Journal of Industrial and Management Optimization} (JIMO). His research expertise is Cybernetic Physics and Autonomous Mobility in general, with a focus on, modeling and  control of aerial robotics with applications, machine learning for dynamic systems and control, visual sensing and machine learning for complex fluids.
\end{IEEEbiography}

\end{document}